\documentclass[12pt,leqno]{article}
\usepackage{amsfonts,amsmath,latexsym, amssymb}
\usepackage[dvips]{graphics}
\usepackage{xcolor}
\usepackage{graphics}

\newcounter{theorem}

\numberwithin{equation}{section}

\newtheorem{theorem}{Theorem}[section]
\newtheorem{corollary}[theorem]{Corollary}
\newtheorem{definition}[theorem]{Definition}

\newtheorem{lemma}[theorem]{Lemma}

\newtheorem{proposition}[theorem]{Proposition}

\newenvironment{proof}[1][Proof]{\textbf{#1.} }
{\ \Box\smallskip}

\newcommand{\RR}{\mathbb{R}}

\newcommand{\half}{\frac{1}{2}}

\newcommand{\real}{\mathbb{R}}

\newcommand{\nat}{\mathbb{N}}

\newcommand{\A}{\mathcal{A}}
\newcommand{\B}{\mathcal{B}}

\newcommand{\cP}{{\cal P}}
\newcommand{\cV}{{\cal V}}
\newcommand{\cW}{{\cal W}}
\newcommand{\cU}{{\cal U}}

\newcommand{\uone}{u_{\tau}^1}
\newcommand{\uzero}{u_{\tau}^0}

\newcommand{\wone}{w_{\tau}^1}
\newcommand{\wzero}{w_{\tau}^0}

\newcommand{\xione}{\xi_{\tau}^1}

\newcommand{\fone}{f_{\tau}^1}

\newcommand{\wktau}{w_{\tau}^k}
\newcommand{\uktau}{u_{\tau}^k}

\newcommand{\fktau}{f_{\tau}^k}

\newcommand{\wktauminus}{w_{\tau}^{k-1}}
\newcommand{\wktauminuss}{w_{\tau}^{k-2}}
\newcommand{\uktauminus}{u_{\tau}^{k-1}}
\newcommand{\uktauminuss}{u_{\tau}^{k-2}}

\newcommand{\wntau}{w_{\tau}^n}
\newcommand{\untau}{u_{\tau}^n}

\newcommand{\xintau}{\xi_{\tau}^n}
\newcommand{\fntau}{f_{\tau}^n}

\newcommand{\wntauminus}{w_{\tau}^{n-1}}
\newcommand{\wntauminuss}{w_{\tau}^{n-2}}
\newcommand{\untauminus}{u_{\tau}^{n-1}}
\newcommand{\untauminuss}{u_{\tau}^{n-2}}

\newcommand{\skalar}[1]{\left\langle #1 \right\rangle}

\newcommand{\dual}[2]{\left\langle #1 \right\rangle_{{#2}^*\times #2}}

\begin{document}

\title{\bf Convergence of a double step scheme for a class of second order Clarke subdifferential inclusions \thanks{\ The research was
		supported by the European
		Union's Horizon 2020 Research and Innovation Programme under
		the Marie Sklodowska-Curie grant agreement No. 823731 CONMECH,
		the Ministry of Science and Higher Education of Republic of Poland
		under Grant Nos. 4004/GGPJII/H2020/2018/0 and 440328/PnH2/2019,
		and the National Science Centre of Poland under Project No. 2021/41/B/ST1/01636.
}}

\author{
	Krzysztof Bartosz,  \
	Pawe{\l} Szafraniec 
	\\ ~ \\
	{\small Jagiellonian University, Faculty of Mathematics and Computer Science} \\
	{\small ul. \L{}ojasiewicza 6, 30348 Krakow, Poland} \\
}

\date{}
\maketitle \thispagestyle{empty}

\vskip 4mm

\noindent {\footnotesize{\bf Abstract.} 
In this paper we deal with a second order evolution inclusion involving a multivalued term generated by a Clarke subdifferential of a locally Lipschitz potential.  For this problem we construct a double step time-semidiscrete approximation, known as the Rothe scheme. We study a sequence of solutions of the semidiscrete approximate problems and provide its weak convergence to a limit element that is a solution of the original problem.

\vskip 2mm

\noindent {\bf 2020 Mathematics Subject Classification:} \ 34G25, 47J22, 49J27, 49J52, 49J53, 65J15, 65N06

\vskip 12mm
\section{Introduction}\label{Sec_1}
Evolutionary problems play very important role in mathematical modeling of various processes in physics, mechanics, economy, biology, etc., where the unknown functions vary in time. Mathematical description of such processes leads usually to differential equations or inclusions, whose type and structure depend on the character of described phenomena and properties of the objects involved. In fact, only very few of them can be solved exactly and in most cases one needs to apply various numerical schemes to achieve an approximate solution. The key idea of solving evolutionary problems is based on the discretization of time interval and constructing a sequence of simple functions related to the length of time step, that converges to an exact solution as the time step converges to zero. In this way, the existence of solution can be proved on one hand and, on the other hand, a behavior and quality of approximate scheme can be analyzed. It gives us a rise to construct various numerical schemes and study their effectiveness.\\
\indent The temporal discretization method, introduced by E. Rothe in 1930 has been used and developed by many authors, see for example \cite{Rektorys} for linear and quasilinear parabolic problems, \cite{Kacur} and \cite{Roubicek2005} for nonlinear parabolic equations and variational inequalities. \\
\indent In this paper we consider a class of second order evolutionary inclusions, in which a multivalued term has a form of Clarke subdifferential of a locally Lipschitz potential. Such inclusions describe a dynamic viscoelastic contact problem in mechanics involving nonmonotone, multivalued frictional contact law. 
The detailed description of the above mechanical process can be found, for instance, in Chapter 5.6 of \cite{HMSBOOK} and in Sections 6-7 of \cite{Bartosz_theta}. It is worth to mention that such processes are typically modeled by means of hemivariational inequalities (HVIs), which represent a weak formulation of a mechanical problems, as specified in Section 6 of \cite{Bartosz_theta}. Furthermore, a solution of HVI can be obtained by solving a related subdifferential inclusion, which  motivates us to deal with inclusions, as the more general problems than HVIs. In Chapter 5 of \cite{HMSBOOK}, a class of second order evolutionary inclusions have been analyzed by means of Rothe method. In that case, the time derivative of an unknown function has been approximated by means of the implicit Euler scheme. This approach has been generalized in \cite{Bartosz_theta} by using the so called $\theta$ scheme, where $\theta\in[0,1]$ is a parameter and, for $\theta=1$, both schemes coincide. Now, we consider a new approach to the class of inclusions studied in  \cite{Bartosz_theta} and \cite{HMSBOOK} . The novelty lies in the use of the so called double step scheme, in which the time derivative of the unknown function at a current point is approximated by a formula involving two previous points. This approach have been successively applied in our previous paper \cite{BSZ} to approximate a solution of a parabolic subdifferential inclusion. Here we extend this result to the class of second order inclusions. Our goal is to study a convergence of the double step Rothe scheme to a solution of the exact problem.\\            
\indent In our work we take advantage of some ideas taken from the previous papers in this field. First of all, we mention \cite{Kalita2013}, where the Rothe method has been applied for the first time to solve evolutionary subdifferential inclusion. 
This pioneering paper, in which a nonlinear parabolic HVI has been studied, inspired many other authors to incorporate Rothe method to solve more general evolutionary HVIs. In particular, it has been applied to second order HVIs in \cite{Peng_Xiao}, to history dependent HVIs in \cite{Migorski_Zeng_1}, to HVIs modeling dynamic contact problems with adhesion in \cite{Migorski_Zeng_2} and to multi-term time fractional integral diffusion equations in \cite{Migorski_Zeng_3}.  Very recently, Rothe method has been used to noncoercive variational-hemivariational inequalities (VHVIs) in \cite{Peng_Huand_Ma} and to history dependent fractional differential HVIs in \cite{Hailing_Cheng}. We also refer to \cite{BCKYZ, Bartosz_Sofonea_1, Kalita} for more applications of Rothe method in various kinds of evolutionary HVIs. 
We also mention that some ideas of our work come from \cite{Emmrich}, where the double step time discretization scheme has been applied to the analysis of a non-Newtonian fluid flow. In a recent paper \cite{LCX} authors anlalyse a similar problem, but with linear operators. Finally, we have intensively exploit \cite{Migorski_Sofonea_book} as a comprehensive review of classical theory and recent advances in the field of variational and hemivariational inequalities. \\  
\indent The rest of the paper is structured as follows. In Section \ref{Sec_2} we introduce notation and present preliminary materials. In Section \ref{Sec_3} we formulate the problem to be solved and state the assumptions on its data. In Section \ref{Sec_4} we introduce the Rothe problem corresponding to the original one, based on a double step semidiscrete approximate scheme. We deal with solvability of Rothe problem and study properties of its solution. Finally, in Section \ref{Sec_5}, we provide a convergence of the sequence of approximate solutions to a solution of the original problem.

\section{Notation and preliminaries}\label{Sec_2}
In this section we introduce notation and recall several known results that will be used in the rest of the paper.

Let $X$ be a real normed space. Everywhere in the paper we will use the symbols $\|\cdot\|_X$, $X^*$ and $\dual{\cdot,\cdot}{X}$ to denote the norm in $X$, its dual space and the duality pairing of $X$ and $X^*$, respectively. Moreover, if $Y$ is a normed space and $f\in {\cal L}(X,Y)$, we will briefly write $\|f\|$ instead of $\|f\|_{{\cal L}(X,Y)}$ and we will use notation $f^*\colon Y^*\to X^*$ for the adjoint operator to $f$. 
We start with the definition of the Clarke generalized directional derivative and the Clarke subdifferential.

\begin{definition}\label{Def_Clarke_subdifferential}
Let $\varphi\colon X\to \RR$ be a locally Lipschitz function. The
Clarke generalized directional derivative of $\varphi$ at the
point $x\in X$ in the direction $v\in X$,  is defined by
\[ \varphi^0(x;v)=\limsup_{y\to x, \lambda \downarrow 0}
\frac{\varphi(y+\lambda v)-\varphi(y)}{\lambda}.\]
The Clarke subdifferential of $\varphi$ at
$x$ is a subset of $X^*$ given by
\[ \partial \varphi(x)=\{\,\xi\in X^*\,\,|\,\, \varphi^0(x;v)\geq
\dual{\xi,v}{X}\,\,\,\,\, \text{for all }\,\,  v\in X\,\}. \]
\end{definition}

\noindent In what follows, we recall the definition of pseudomonotone operator in both single-valued and multivalued cases (see \cite{Zeidler1990} and \cite{Migorski_Sofonea_book}).

\begin{definition}\label{Def_pseudomonotone_single}
 A single-valued operator $A\colon X\to X^*$ is called pseudomonotone
 if for any sequence $\{v_n\}_{n=1}^\infty\subset X$, $v_n\to v$
 weakly in $X$ and $\limsup_{n\to\infty}\dual{Av_n, v_n-v}{X}\leq
 0$
 imply that
 $\dual{Av, v-y}{X} \leq \liminf_{n\to\infty}\dual{Av_n, v_n-y}{X}$
 for every $y\in X$.
\end{definition}

\begin{definition}\label{Def_pseudomonotone_multi}
	A multivalued operator $A\colon X\to 2^{X^*}$ is called pseudomonotone if the following conditions hold:
	\begin{itemize}
		\item[ 1)] $A$ has values which are nonempty, bounded, closed and convex.
		\item[ 2)] $A$ is upper semicontinuous (usc, in short) from every finite dimensional subspace of $X$
		into $X^{\ast}$ endowed with the weak to\-po\-lo\-gy.
		\item[ 3)] For any sequence $\{v_n\}_{n=1}^\infty\subset X$ and
		any $v_n^{\ast} \in A(v_n)$, $v_n \rightarrow v$ weakly in $X$ and\\ $ \limsup_{n\rightarrow
			\infty} \dual{v_n^\ast, v_n - v}{X}\leq 0$ imply that for
		every $y \in X$ there exists $u(y) \in A(v)$ such that $\dual{
			u(y), v - y}{X}
		\leq \liminf_{n \rightarrow \infty} \dual{ v_n^{\ast}, v_n - y}{X}$.
	\end{itemize}
\end{definition}

\noindent The following two propositions provide an important class of pseudomonotone operators that will appear in the next section.
They correspond to Proposition 5.6 in \cite{HMSBOOK} and Proposition
1.3.68 in \cite{DMP2}, respectively.

\begin{proposition}\label{Prop_2.1}
	Let $X$ and $U$ be two reflexive Banach spaces and $\iota\colon  X\to
	U$ a linear, continuous and compact operator. Let $J\colon
	U\to\real$ be a locally Lipschitz functional and  assume that its Clarke
	subdifferential satisfies
	\begin{equation}\nonumber
	\|\xi\|_{U^*}\leq c(1+\|u\|_U)\quad\text{for all}\,\,\, u\in U,\,\,\xi\in\partial J(u)
	\end{equation}
	with $c>0$. Then the multivalued operator $M\colon X\to 2^{X^*}$ defined by
	\begin{equation}\nonumber
	M(v)=\iota^*\partial J(\iota v) \quad \text{for all}\,\,\,  v\in X
	\end{equation}
	is pseudomonotone.
\end{proposition}

\begin{proposition}\label{Prop_sum_pseudo}
	Assume that  $X$ is a reflexive Banach space and $A_1, A_2\colon X\to
	2^{X^*}$ are pseudomonotone operators. Then the operator $A_1+A_2\colon X\to 2^{X^*}$
	is pseudomonotone.
\end{proposition}
    
\noindent In what follows, we introduce the notion of coercivity.
\begin{definition}
	Let $X$ be a real Banach space and $A\colon X\to 2^{X^*}$ be a multivalued operator. We say that $A$ is coercive if either $D(A)$ is bounded or $D(A)$ is unbounded and
	$$
	\lim_{\|v\|_X\to\infty\,\,v\in D(A)}\frac{\inf\{\dual{v^*,v}{X}\,|\,\,v^*\in A v\}}{\|v\|_X}=+\infty,
	$$
	where, recall, $D(A)=\{x\in X|\,\,A(x)\neq\emptyset\}$ is the domain of $A$.
\end{definition}
The following is the main surjectivity result for multivalued  pseudomonotone operator.

\begin{proposition}\label{prop:Bartosz5}
	Let $X$ be a real, reflexive Banach space and $A\colon X\to 2^{X^*}$ be  bounded,  coercive and pseudomonotone. Then $A$ is surjective, i.e., for all $b\in X^*$ there exists $v\in X$ such that $Av\ni b$.
\end{proposition}

\noindent We also recall Lemma 3.4.12 of \cite{DMP2}, known as Ehrling's Lemma.

\begin{lemma}\label{Lemma_Ehrling}
	If $V_0,\, V,\, V_1$ are Banach spaces such that $V_0\subset V\subset V_1$, the embedding of $V_0$ into $V$ is compact and the embedding of $V$ into $V_1$ is continuous, then for every $\varepsilon>0$, there exists a constant $c(\varepsilon)>0$ such that
	\begin{eqnarray}
	\|v\|_V\leq\varepsilon\|v\|_{V_0}+c(\varepsilon)\|v\|_{V_1}\quad \text{for all}\,\,\,v\in V_0.
	\end{eqnarray}
\end{lemma}

\indent Let $X$ be a Banach space and $T>0$.
We introduce the space $BV(0,T;X)$ of functions of
bounded total variation on $[0,T]$. Let $\pi$ denote any finite
partition of $[0,T]$ by a family of disjoint subintervals $\{\sigma_i
= (a_i,b_i)\}$ such that $[0,T] = \cup_{i=1}^n \bar{\sigma}_i$.
Let $\mathcal{F}$ denote the family of all such partitions. Then,
for a function $x\colon [0,T]\to X$ and for $1 \leq q < \infty$,
we define a seminorm 
\begin{equation}\nonumber
\| x \|^q_{BV^q (0,T;X)} = \sup_{\pi \in \mathcal{F}} \left\{
\sum_{\sigma_i \in \pi} \|x(b_i)-x(a_i)\|_X^q \right\},
\end{equation}
and the space
$$
BV^q(0,T;X)=\{x\colon [0,T]\to X |\,\,\| x \|_{BV^q (0,T;X)}<\infty\}.
$$
For $1 \leq p \leq \infty$, $1\leq q < \infty$ and Banach spaces $X$, $Z$ such that $X
\subset Z$, we introduce a vector space
$$ M^{p,q}(0,T;X,Z) = L^p(0,T;X) \cap BV^q (0,T;Z).$$
Then $M^{p,q}(0,T; X,Z)$ is also a Banach space with the norm given by $\| \cdot \|_{L^p(0,T;X)} + \| \cdot \|_{BV^q(0,T;Z)}.$\\

\noindent The following proposition will play the crucial role for the convergence of the Rothe functions which will be constructed later. 
For its proof, we refer to \cite{Kalita2013}. 

\begin{proposition}\label{prop:Bartosz6}
	Let $1\leq p,q<\infty$. Let $X_1\subset X_2\subset X_3$ be real Banach spaces such that $X_1$ is reflexive, the embedding $X_1\subset X_2$ is compact and the embedding $X_2\subset X_3$ is continuous. Then the embedding $M^{p,q}(0,T;X_1;X_3)\subset L^p(0,T;X_2)$ is compact. 
\end{proposition}

\noindent The following version of Aubin-Celina  convergence theorem (see \cite{Aubin}) will be used in what follows.
\begin{proposition}\label{prop:Bartosz7}
	Let $X$ and $Y$ be Banach spaces and $F\colon X\to 2^Y$ be a multifunction such that
	\begin{itemize}
		\item[(a)] the values of $F$ are nonempty, closed and convex subsets of $Y$,
		\item[(b)] $F$ is upper semicontinuous from $X$ into $w-Y$.
	\end{itemize}
Let $x_n,x\colon(0,T)\to X$, $y_n,y\colon(0,T)\to Y$, $n\in \mathbb{N}$, be measurable functions such that $x_n(t)\to x(t)$ for a.e. $t\in (0,T)$ and $y_n\to y$ weakly in $L^1(0,T;Y)$. If $y_n(t)\in F(x_n(t))$ for all $n\in\mathbb{N}$ and a.e. $t\in (0,T)$, then $y(t)\in F(x(t))$ for a.e. $t\in (0,T)$.
\end{proposition}

\noindent We recall a well known Young's inequality
\begin{equation}\label{Young}
ab\leq\varepsilon a^2+\frac{1}{4\varepsilon}b^2,
\end{equation}
for all $a,b\in\real$, $\varepsilon>0$.\\

\noindent We conclude this section with the following lemma, which is trivial, yet crucial in what follows.
\begin{lemma}\label{BWK} For a sequence $\{\delta_i\}_{i=1}^N\subset X$ it holds that
	\[
	\sum_{i=1}^N \|\delta_i\|_{X}^2 \le c\left(\|\delta_1\|_{X}^2 + \sum_{i=2}^N \left\|\frac32\delta_i-\frac12\delta_{i-1}\right\|_{X}^2\right)
	\]
with $c>0$.	
\end{lemma}
\begin{proof} A simple calculation shows that
	\[
	\sum_{i=1}^N \|\delta_i\|_{X}^2\le \sum_{i=1}^N \left\|\frac{3}{2}\delta_i\right\|_{X}^2\le \left\|\frac32\delta_1\right\|_{X}^2 + 2\sum_{i=2}^N \left\|\frac32 \delta_i -\frac12\delta_{i-1}\right\|_{X}^2 + 2\sum_{i=1}^{N-1} \left\|\frac12\delta_i\right\|_{X}^2.
	\]
	Hence, it follows that
	\[
	\frac12 \sum_{i=1}^{N-1}\|\delta_i\|_{X}^2 + \|\delta _N\|_{X}^2\le \frac94\|\delta _1\|_{X}^2 + 2\sum_{i=2}^N \left\|\frac32\delta_i-\frac12 \delta_{i-1}\right\|_{X}^2,
	\]
	which yields the assertion. 
\end{proof}

\section{Problem formulation}\label{Sec_3}
In this section we formulate an abstract second order evolution inclusion involving Clarke subdifferential. We also impose assumptions on the data of the problem.\\
\indent Let $V$ be a real, reflexive, separable Banach space and $H$ be a real, separable Hilbert space equipped with  the inner product $(\cdot,\cdot)_H$ and the corresponding norm given by  $\|v\|_H=\sqrt{(v,v)_H}$ for all $v\in H$. For simplicity of notation, we will write  $|v|=\|v\|_H$, $(u,v)=(u,v)_H$ for all $u,v\in H$ and $\|v\|=\|v\|_V$, $\skalar{l,v}=\dual{l,v}{V}$ for all $v\in V$, $l\in V^*$.  Identifying $H$ with its dual we assume that the spaces $V,\, H$ and $V^*$ form an evolution triple, i.e., $V\subset H\subset V^*$ with all embeddings being dense and continuous. Moreover, we assume that the embedding $V\subset H$ is compact. Let $i\colon V\to H$ be an embedding operator (for $v\in V$ we still denote $iv\in H$ by $v$). For all $u\in H$ and $v\in V$ we have $\skalar{u,v}=(u,v)$. We also introduce a reflexive Banach space $U$ and the operator $\iota\in {\cal L}(V,U)$. For $T>0$ we denote by $[0,T]$ a time interval and introduce the following spaces of time dependent functions: ${\cal V}=L^2(0,T;V)$, ${\cal V^*}=L^2(0,T;V^*)$, ${\cal H}=L^2(0,T;H)$, ${\cal U}=L^2(0,T;U)$, ${\cal U^*}=L^2(0,T;U^*)$, equipped with their classical $L^2$ norms. We use notation $\dual{u,v}{\cal V}=\int_0^T\skalar{u(t), v(t)}dt$, for all $u\in {\cal V}^* ,v\in {\cal V}$, $(u,v)_{\cal H}=\int_0^T(u(t), v(t))dt$ for all $u,v\in {\cal H}$ and $\dual{u,v}{\cal U}=\int_0^T\dual{u(t), v(t)}{U}dt$ for all $u\in {\cal U}^*,v\in {\cal U}$. Finally, we define the space ${\cal W}=\{v\in {\cal V}\,\, |\,\,v'\in {\cal V^*}\}$. Hereafter, $v'$ denotes the time derivative of $v$ in the sense of distribution. \\
\indent We consider the operators $A, B\colon V\to V^*$ and the functions $j\colon U\to \real$, $f\colon [0,T]\to V^*$. Using the above notation we formulate the following problem.

\vskip3mm \noindent {\bf Problem ${\cal P}$}. {\it Find $u\in\cV$, with $u'\in \cW$ such that $u(0)=u_0$, $u'(0)=w_0$ and
	\begin{align}
	& u''(t)+Au'(t)+ Bu(t) +\iota^*\xi(t)=f(t)\quad \text{\it for\ a.e.}\
	\,t\in (0,T)\label{3.1}
	\end{align}
	\vspace{-1mm} with
	\begin{equation}
	\xi(t)\in \partial j(\iota u'(t))\quad \text{\it for\ a.e.}\
	\,t\in (0,T).\label{3.2}\end{equation} }

\noindent We now impose assumptions on the data of Problem  $\cP$. \\

\noindent {$H(A)$} The operator $A\colon V\to V^*$ satisfies
\begin{itemize}
	\item[(i)] $A$ is pseudomonotone,
	\item[(ii)] $\|Av\|_{V^*}\leq a+b\|v\|$ for all $v\in V$, with $a\geq 0$, $b>0$,
	\item[(iii)] $\skalar{Av,v}\geq \alpha\|v\|^2-\beta|v|^2$
	 for all $v\in V$ with $\alpha>0$, $\beta\geq 0$.
\end{itemize}

\medskip
\noindent {$H(B)$} The operator $B\colon V\to V^*$ is linear, bounded, symmetric and positive.

\medskip
\noindent {$H(j)$} The functional $j\colon U\to \real$ is
such that
\begin{itemize}
\item [(i)] $j$ is locally  Lipschitz\vspace{-1mm},
\item [(ii)] $\partial
j$ satisfies the following growth condition\vspace{-1mm}
\begin{align*}\| \xi \|_{U^{\ast}} \leq d(1+\|u\|_U) 
\end{align*}
for all $u\in U,\, \xi\in \partial j(u)$ with $d>0$.
\end{itemize}

\noindent {$H(\iota)$} The operator $\iota\colon V
\rightarrow U$ is linear, continuous and compact. Moreover,
there exists a Banach space $Z$ such that $V\subset Z\subset H$, where the embedding $V\subset Z$ is compact, the embedding $Z\subset H$ is continuous and the operator $\iota$ can be decomposed as $\iota=\iota_2\circ\iota_1 $, where $\iota_1\colon V\to Z$ denotes the (compact) identity mapping and  $\iota_2\in{\cal L}(Z,U)$. \\


\noindent {$H(f)$} $f \in L^2(0,T;V^*)$.\\

\noindent {$H(0)$}  $u_0 \in V$, $w_0\in H$.\\ 

We remark that Problem ${\cal P}$ represents an abstract formulation of a dynamic viscoelastic contact process between a physical body and a foundation. In particular, $A$ and $B$ are related to viscosity and elasticity operators, respectively and $f$ represents the density of a volume force acting on the body. Furthermore, the presence of the multivalued term $\partial j$ follows from the use of the nonsmooth and potentially nonmonotone friction law in the model. Finally, $\iota$ is related to the trace operator. For the precise description of the physical process we refer to Sections 6-7 of \cite{Bartosz_theta}, for instance.\\      

\noindent In the rest of the paper we always assume that assumptions $H(A)$, $H(B)$, $H(j)$, $H(\iota)$, $H(f)$ and $H(0)$ hold.

\begin{corollary}\label{Cor_3.1}
	For every $\varepsilon>0$ there exists $c_{\iota}(\varepsilon)>0$ such that
	\begin{align}\label{3.3}
	&\|\iota u\|_U\leq\varepsilon\|u\|+c_{\iota}(\varepsilon)|u|\quad\text{for all}\,\,\,u\in V.
	\end{align}
\end{corollary}
\begin{proof}	
	It follows from $H(\iota)$ that $\|\iota u\|_U=\|\iota_2\circ \iota_1 u\|_U\leq \|\iota_2\|\|\iota_1 u\|_Z$. On the other hand, from Lemma \ref{Lemma_Ehrling}, we have $\|\iota_1 u\|_Z\leq \varepsilon\|u\|+c(\varepsilon)|u|$, which yields the assertion.
\end{proof}
\medskip 

\noindent The direct consequence of Corollary \ref{Cor_3.1} is following.
\begin{corollary}\label{Cor_3.2}
	For every $\varepsilon>0$, we have
	\begin{align}\label{3.4}
	&\|\iota u\|^2_U\leq\varepsilon\|u\|^2+\bar{c}_\iota(\varepsilon)|u|^2,
	\end{align}
	where $\bar{c}_\iota(\varepsilon)=2c_\iota^2(\sqrt{\varepsilon/2})$.
\end{corollary}

\section{The Rothe problem}\label{Sec_4}

In this section we consider a semidiscrete approximation of Problem $\cP$ known as Rothe problem. Our goal is to study a solvability of the Rothe problem, to obtain a-priori estimates for its solution and to study the convergence of semidiscrete solution to the solution of the original problem as the discretization parameter converges to zero.\\
We start with a uniform division of the time interval. Let $N\in \nat$ be fixed and $\tau=T/N$ be a time step and write $t_n=n\tau$ for $n=0,\dots,N$. In the rest of the paper we denote by $C$ a generic positive constant independent on discretization parameters, that can differ from place to place. Moreover, if $C$ depends on $\varepsilon$, we write $C(\varepsilon)$.\\ 
Let $u_\tau^0, \,\wzero \in V$ be given and assume that
\begin{align}
&u_\tau^0\to u_0 \,\,\,\,\,\, \mbox{strongly in} \ V, \label{uzero}\\[2mm]
&\wzero \to w_0 \,\,\,\, \mbox{strongly in} \ H  \ \mbox{and weakly in} \ V \label{wzero}
\end{align}
as $\tau\to 0$ and
\begin{equation}\label{-1}
\|w_\tau^0\|\leq\frac{C}{\sqrt{\tau}} \,\,\,\, \mbox{for all} \ \tau>0.
\end{equation}

We now describe a concept of semidiscrete approximation scheme to be used in the analysis of Problem $\cP$. Suppose that $x\colon[0,T]\to X$ is a differentiable function and $x^n$ denotes its value  at  $t_n$ for $n=0,\dots,N$. Then, we approximate the derivative of $x$ at $t_n$ by the following formula
\begin{align}\label{double_step}
	x'(t_n)\simeq \frac{1}{\tau} \left( \frac{3}{2}x^n-2x^{n-1}+\frac{1}{2}x^{n-2} \right) \,\,\,\, \mbox{for} \,\,n=2,\dots, N. 
\end{align}
The formula (\ref{double_step}) is referred to as a double step approximation scheme. It is routine to check that (\ref{double_step}) becomes exact if $x$ is a polynomial of order $\leq 2$. In a general case, (\ref{double_step}) reflects the fact that $x'(t_n)$ is approximated by the derivative at the point $t_n$ of the unique polynomial of order $\leq 2$ interpolating the function $x$ at points $t_n, t_{n-1}$ and $t_{n-2}$. Note that (\ref{double_step}) involves two points preceding $t_n$, hence, it is not applicable to approximate $x'(t_1)$. In that case we can use, for instance, the implicit Euler scheme $x'(t_1)\simeq \frac{1}{\tau}(x^1-x^0)$.\\
Using the above idea, we approximate the right hand side of (\ref{3.1}). Let $F(t) = \int_0^t f(s)\, ds$. Then, we write $f(t)=F'(t)$ and 
introduce a sequence $\{\fntau\}_{n=1}^N\subset V^*$ approximating $f(t_n)$, $n=1,\dots,N$ as follows
\begin{align*}
f(t_1)=F'(t_1)\simeq f_\tau^1:=\frac{1}{\tau}(F(t_1)-F(t_0))=\frac{1}{\tau}\int_0^\tau f(t)dt	
\end{align*} 
and, for $n=2,\dots,N$,
\begin{align*}
	f(t_n) = F'(t_n)\simeq \fntau&:=\frac{1}{\tau}\bigg(\frac32 F(t_n)-2F(t_{n-1}) + \frac12 F(t_{n-2})\bigg) \\[2mm]
	& =\frac{1}{\tau}\bigg(\frac32 \int_0^{t_n} f(s)\, ds - 2\int_0^{t_{n-1}} f(s)\, ds + \frac12 \int_0^{t_{n-2}} f(s)\, ds\bigg)   \\[2mm]
	&=\frac{1}{\tau} \bigg(\frac32 \int_0^{t_n} f(s)\, ds - \frac32 \int_0^{t_{n-1}} f(s)\, ds - \frac12\int_0^{t_{n-1}} f(s)\, ds + \frac12 \int_0^{t_{n-2}} f(s)\, d\bigg)\\[2mm]
 &=\frac{3}{2\tau} \int_{t_{n-1}}^{t_n} f(s)\, ds -\frac{1}{2\tau} \int_{t_{n-2}}^{t_{n-1}} f(s)\, ds.
\end{align*}

Next, for an unknown function $u$, we write $u'=w$, hence, $u''=w'$. Our goal is to find sequences $\{\untau\}_{n=1}^N,\,\{\wntau\}_{n=1}^N\subset V$ approximating 
$u(t_n)$ and $w(t_n)$, respectively, for  $n=1,\dots,N$. To this end, we formulate the following semidiscrete double step Rothe problem corresponding to Problem $\cP$.

\bigskip \noindent{\bf Problem ${\cal P}_\tau$}. {\it Find 
	sequences \  $\{\untau\}_{n=1}^N,\,\, \{\wntau\}_{n=1}^N\subset V$ such that
	\begin{align}
	&\frac{1}{\tau} (\wone-\wzero) + A\wone + B\uone + \iota^*\xione =\fone, \label{1}\\[2mm]
	& \uone=\tau \wone+\uzero, \label{1b}\\[2mm]
	&\xione \in \partial j(\iota\wone). \label{1a}
	 	\end{align}}
and for $n=2,...,N$ 
\begin{align}
&\frac{1}{\tau}\left(\frac32 \wntau-2\wntauminus + \frac12 \wntauminuss\right) +  A\wntau +  B\untau +  \iota^*\xintau =  \fntau, \label{2}\\[2mm]
& \untau=\frac23\left(\tau\wntau+2\untauminus-\frac12\untauminuss\right),\label{2b}\\[2mm]
&	\xintau \in \partial j (\iota\wntau).\label{2a}
\end{align} 	

The first term of (\ref{1}) approximates $w'(t_1)$ by means of the implicit Euler scheme. Moreover, (\ref{1b}) is equivalent with $\wone=\frac{1}{\tau}(\uone-\uzero)$, which reflects the fact that $w(t_1)=u'(t_1)$ is approximated in the same way. The first term of (\ref{2}) approximates $w'(t_n)$ by means of the double step scheme (\ref{double_step}). Moreover, (\ref{2b}) is equivalent with $\wntau=\frac{1}{\tau}(\frac{3}{2}u_\tau^n-2u_\tau^{n-1}+\frac{1}{2}u_\tau^{n-2})$, which reflects the fact that $w(t_n)=u'(t_n)$ is approximated by formula (\ref{double_step}) for $n=2,\dots,N$.\\ 
  
We now deal with the solvability of Problem $\cP_\tau$.

\begin{theorem}\label{Th_4.1}
Under assumptions $H(A)$, $H(B)$, $H(j)$, $H(\iota)$, $H(f)$ and $H(0)$ there exists $\tau_0>0$ such that for all $0<\tau<\tau_0$, Problem $\cP_\tau$ has a solution.   
\end{theorem}

\noindent Before we prove Theorem \ref{Th_4.1}, we construct two auxiliary multivalued operators $T_0, T\colon V\to 2^{V^*}$  given by
\begin{align*}
&T_0v = i^*i v + \tau Av + \tau B(\tau v ) + \tau \iota^*\partial j(\iota v)\quad\quad\quad \text{for all}\,\,\, v\in V,\\
&Tv=\frac32i^*iv +\tau Av + \tau B\left(\frac23 \tau v\right)+ \tau \iota^*\partial j(\iota v)\quad \text{for all}\,\,\, v\in V,
\end{align*}
where $\tau>0$ is given.\\

\indent We now provide three lemmata on the properties of operators $T_0$ and $T$. 
\begin{lemma}\label{lemma_T_coercive}
Under the assumptions $H(A)$, $H(B)$, $H(j)$ and $H(\iota)$ there exists $\tau_0>0$, such that for all $0<\tau<\tau_0$ operators $T_0$ and $T$ are coercive. 
\end{lemma}
\begin{proof}
We restrict the proof to the case of operator $T_0$ and remark that the proof of coercivity of operator $T$ is analogous.	
Let $v\in V$ and let $v^*\in T(v)$. It means that $v^*=i^*i v + \tau Av + \tau B(\tau v) + \tau\iota^*\xi$, with $\xi\in\partial j(\iota v)$. Our goal is to estimate $\skalar{v^*,v}$.  
We start with the last term. Namely, by $H(j)$(ii) and Corollary \ref{Cor_3.2}, we get
\begin{align}\label{poprawki_15}
&	\nonumber \tau \langle \iota^*\xi,v\rangle  = \tau \langle \xi, \iota v\rangle_{U^*\times U} \ge -\tau \|\xi\|_{U^*}\|\iota v\|_U  \\[2mm]
&	\ge \nonumber -\tau d(1+\|\iota v\|_U)\|\iota v\|_U = -\tau d \|\iota v\|_U - \tau d\|\iota v\|_U^2  \\[2mm]
&	\ge \nonumber -\tau\left(\frac12 \|\iota v\|_U^2 + \frac12 d^2\right) - \tau d\|\iota v\|_U^2 = -\tau \left(\frac12 + d\right)\|\iota v\|_U^2 -  \frac12\tau d^2  \\[2mm]
&	\ge -\tau\left(\frac12 +d\right) \varepsilon \|v\|^2 - \tau\left(\frac12 +d\right) \bar{c}_\iota(\varepsilon) |v|^2 -\frac12 \tau  d^2.	
\end{align} 	
for all $\varepsilon>0$.
Next, using $H(A)$(iii), $H(B)$ and (\ref{poprawki_15}), we obtain
\begin{align*}
&\skalar{T_0 v^*, v}=|v|^2 + \tau \skalar{Av,v} + \tau^2 \skalar{B v,v} +\tau \langle \iota^*\xi,v\rangle\\[2mm]
&\geq\left(1-\beta-\tau\left(\frac12+d\right) \bar{c}_\iota(\varepsilon) \right)|v|^2+\tau\left(\alpha-\left(\frac12+d\right)\varepsilon\right)\|v\|^2-\frac12 \tau  d^2.
\end{align*}
We take $\varepsilon$ small enough such that the coefficient at $\|v\|^2$ is positive. Then its clear that $T_0$ is coercive for every $\tau>0$ for which the coefficient at $|v|^2$ is positive. 
\end{proof}	

\begin{lemma}\label{lemma_T_bounded}
	Under the assumptions $H(A)$, $H(B)$, $H(j)$ and $H(\iota)$ operators $T_0$ and $T$ are bounded. 
\end{lemma}
\begin{proof}
	Clearly, the operator $V\ni v\to i^*iv\in V^*$ is bounded, as $\|i^*iv\|_{V^*}\leq\|i\|^2\|v\|$. The boundedness of operator $\tau A$ and the operator $V\ni v\to \tau B(\tau v)\in V^*$ follows from assumptions $H(A)$(ii) and $H(B)$, respectively. Finally, the operator $V\ni v\to \tau\iota^*\partial j(\iota v)\in 2^{V^*}$ is bounded by assumptions $H(j)$(ii) and $H(\iota)$. Hence, $T_0$ is bounded as a sum of bounded operators. The boundedness of $T$ follows from analogous arguments. 
\end{proof}

\begin{lemma}\label{lemma_T_pseudomonotone}
	Under the assumptions $H(A)$, $H(B)$, $H(j)$ and $H(\iota)$ operators $T_0$ and $T$ are pseudomonotone. 
\end{lemma}
\begin{proof}
	 The operators $V\ni v\to i^*iv\in V^*$ and $V\ni v\to \tau B\tau v\in V^*$ are pseudomonotone, as they are linear and monotone. The operator $\tau A$ is pseudomonotone by assumption $H(A)$(i). Finally, the operator $V\ni v\to \tau\iota^*\partial j(\iota v)\in 2^{V^*}$ is pseudomonotone by assumptions $H(j)$(ii), $H(\iota)$ and Proposition \ref{Prop_2.1}. Hence, $T_0$ is pseudomonotone as a sum of pseudomonotone operators (see Proposition \ref{Prop_sum_pseudo}). The pseudomonotonicity of $T$ follows from analogous arguments. 
\end{proof}
\medskip
\medskip

\indent We are now in a position to present the proof of Theorem \ref{Th_4.1}.\\

\noindent \begin{proof} From lemmata \ref{lemma_T_coercive}-\ref{lemma_T_pseudomonotone} and Proposition \ref{prop:Bartosz5} it follows that operators $T_0$ and $T$ are surjections. Hence, in particular, there exists $\wone\in V$, such that
	\begin{equation}\label{poprawki_17} 
	\tau \fone + i^*i\wzero -\tau Bu_\tau^0\in T_0\wone.
	\end{equation}
	Using $\wone$ derived from (\ref{poprawki_17}), we recover $\uone$ from (\ref{1b}) and observe that (\ref{poprawki_17})
	 is equivalent with \eqref{1}, in which the element $u_\tau^1$ is given by (\ref{1b}) and $\xi_\tau^1$ satisfies (\ref{1a}). 
	Suppose that we have already given $\untauminuss$, $\untauminus$, $\wntauminuss$ and $\wntauminus$ for $n=2,\ldots,N$.  Since $T$ is surjection, there exists $w_\tau^n\in V$, such that
	\begin{equation}\label{poprawki_18}
	\tau \fntau + 2 i^*i\wntauminus -\frac12 i^*i\wntauminuss -\tau B\left(\frac43 \untauminus-\frac13\untauminuss\right) \in T\wntau.
	\end{equation}	
Using $\wntau$ derived from (\ref{poprawki_18}), we recover $\untau$ from (\ref{2b}) and observe that (\ref{poprawki_18})
is equivalent with \eqref{2}, in which the element $u_\tau^n$ is given by (\ref{2b}) and $\xi_\tau^n$ satisfies (\ref{2a}). \\
In this way, we construct the entire sequences  $\{\untau\}_{n=1}^N, \{\wntau\}_{n=1}^N$ solving Problem $\cP_\tau$, which completes the proof. 
\end{proof}\\


 We now deal with a-priori bounds for the solution of Problem $\cP_\tau$.
\begin{lemma}\label{Lemma_apriori}
Let $\{\untau\}_{n=1}^N$, $\{\wntau\}_{n=1}^N$ be a solution of Problem $\cP_\tau$ and let $\{\xintau\}_{n=1}^N\subset U^*$ be a corresponding sequence satisfying (\ref{1a}) and (\ref{2a}). Then, we have
\begin{align}
&\label{a1}\tau \sum_{n=0}^N \|\wntau\|^2\le C,\\[2mm]
&\label{a2}\max_{n=0,...,N}|\wntau|<C,\\[2mm]
&\label{a3}\tau \sum_{n=1}^N \left\|\xintau\right\|_{U^*}^2\le C,\\[2mm]
&\label{a4}\tau \left\|\frac{w_\tau^1-w_\tau^0}{\tau}\right\|_{V^*}^2 \le C,\\[2mm]
&\label{a5}\tau \sum_{n=2}^N\left\|\frac{1}{\tau}\left(\frac{3}{2}\wntau -2\wntauminus +\frac{1}{2}\wntauminuss\right)\right\|_{V^*}^2 \le C,\\[2mm]
&\label{a6}\sum_{n=2}^N \left|\wntau-2\wntauminus+\wntauminuss\right|^2\le C, \\[2mm]
&\label{a7}\max_{n=0,...,N}\|\untau\|<C.
\end{align}
\end{lemma}

\noindent \begin{proof}
We test equation (\ref{1}) with $w_\tau^1$ and obtain
\begin{align}\label{poprawki_1}
(\wone-\wzero,\wone) + \tau\langle A\wone,\wone\rangle + \tau\langle B\uone,\wone\rangle + \tau\langle \iota^*\xione,\wone\rangle  =\tau\langle \fone,\wone\rangle. 
\end{align}

\noindent We now deal with each term of (\ref{poprawki_1}). By elementary properties of inner product we get
\begin{align}
	\label{ap1}(\wone -\wzero,\wone) = \frac12 |\wone|^2 - \frac12|\wzero|^2 +\frac12|\wone -\wzero|^2.
\end{align}
Next, from $H(A)$(iii) and $H(B)$, we obtain 
\begin{align}
&	\label{ap2}\tau\langle A\wone,\wone\rangle \ge \tau \alpha \|\wone\|^2-\tau \beta |\wone|^2, \\[2mm]
& \label{ap3}\tau \langle B\uone, \wone\rangle 
	=\langle B\uone,\uone - \uzero\rangle  = \frac{1}{2}\langle B\uone,\uone \rangle  - \frac12 \langle B\uzero,\uzero\rangle + \frac12 \langle B(\uone-\uzero),\uone-\uzero\rangle.
\end{align}
From $H(j)$(ii), similarly to (\ref{poprawki_15}), we get 
\begin{align}
&	\label{ap4} \tau \langle \iota^*\xione,\wone\rangle \ge -\tau\left(\frac12 +d\right) \varepsilon \|\wone\|^2 - \tau\left(\frac12 +d\right) \bar{c}_\iota(\varepsilon) |\wone|^2 -\frac12 \tau  d^2.
\end{align}
Finally,  from $H(f)$, we have
\begin{align}
&	\label{ap5}\tau \langle \fone, \wone \rangle \le \tau \|\fone\|_{V^*} \|\wone\| \le \varepsilon \tau \|\wone\|^2 + C(\varepsilon)\tau \|\fone\|_{V^*}^2.
\end{align}
We multiply \eqref{poprawki_1} by $\frac12$ and apply \eqref{ap1}--\eqref{ap5} to the resulting equality. Then, we obtain
\begin{align}
&	\nonumber \frac14 |\wone|^2 + \frac14 |\wone-\wzero|^2 	+ \frac12\tau\left(\alpha - \left(\frac32 +d\right)\varepsilon\right) \|\wone\|^2 + \frac14 \langle B\uone,\uone\rangle   \\[2mm]
&\nonumber + \frac14 \langle B(\uone-\uzero),\uone-\uzero\rangle \le \frac14 |\wzero|^2 + \frac14 \langle B\uzero,\uzero\rangle\\[2mm]
&	 +\frac12\tau \left(\beta +\left(\frac12 +d\right)\bar{c}_{\iota}(\varepsilon) \right) |\wone|^2+ \frac14 \tau d^2 + \frac12C(\varepsilon) \tau \|\fone\|_{V^*}^2\label{ap6}.
\end{align}
We test equation (\ref{2}) with $w_\tau^n$ and obtain
\begin{align}\label{poprawki_1a}
	\left(\frac32 \wntau-2\wntauminus + \frac12 \wntauminuss,\wntau\right) + \tau\langle A\wntau,\wntau\rangle + \tau\langle B\untau,\wntau\rangle + \tau\langle \iota^*\xintau,\wntau\rangle  =\tau\langle \fntau,\wntau\rangle. 
\end{align}
\noindent By elementary properties of inner product we get
\begin{align}
&\nonumber \left( \frac32 \wntau - 2\wntauminus + \frac12 \wntauminuss, \wntau\right)  \\[2mm]
&\label{ap8}=\frac14 \left(|\wntau|^2 + |2\wntau-\wntauminus|^2 -|\wntauminus|^2 - |2\wntauminus-\wntauminuss|^2 + |\wntau-2\wntauminus+\wntauminuss|^2\right). 
\end{align}

\noindent From $H(A)$(iii) and $H(B)$ we get
\begin{align}
&\label{ap9}\tau\langle A\wntau,\wntau\rangle \ge \tau \alpha \|\wntau\|^2 - \tau\beta|\wntau|^2, \\[2mm]
	&	\nonumber \tau \langle B\untau,\wntau\rangle = \langle B\untau,\frac32\untau-2\untauminus + \frac12 \untauminuss \rangle = \frac14 \langle B\untau,\untau\rangle  \\[2mm]
	&	 + \frac14\langle B(2\untau-\untauminus),2\untau-\untauminus\rangle - \frac14\langle B\untauminus,\untauminus\rangle-\frac14\langle B(2\untauminus-\untauminuss),2\untauminus -\untauminuss \rangle \nonumber  \\[2mm]
	&	\label{ap7} + \frac14\langle B(\untau-2\untauminus+\untauminuss), \untau-2\untauminus+\untauminuss\rangle. 
\end{align}
From  $H(j)$(ii) and $H(f)$
\begin{align}
&\label{ap10}\tau \langle \iota^*\xintau,\wntau\rangle \ge -\tau \left(\frac12+d\right)\varepsilon \|\wntau\|^2 - \tau\left(\frac12+d\right)\bar{c}_{\iota}(\varepsilon)|\wntau|^2 - \tau\frac12d^2, \\[2mm]
&\label{ap11}\tau \langle \fntau,\wntau\rangle \le \varepsilon \tau \|\wntau\|^2 + C(\varepsilon)\tau \|\fntau\|_{V^*}^2.	
\end{align}

\noindent Applying \eqref{ap8}--\eqref{ap11} to \eqref{poprawki_1a}, we obtain
\begin{align}
&	\nonumber \frac14 |\wntau|^2 + \frac14 |2\wntau-\wntauminus|^2 + \frac14 |\wntau - 2\wntauminus + \wntauminuss|^2+\frac14 \langle B\untau,\untau\rangle   \\[2mm]
&	 \nonumber + \frac14 \langle B(2\untau-\untauminus),2\untau-\untauminus\rangle  + \frac14 \langle B(\untau-2\untauminus+\untauminuss),\untau-2\untauminus+\untauminuss\rangle \\[2mm]
&	\nonumber + \tau \left(\alpha - \left(\frac32 +d\right)\varepsilon \right)\|\wntau\|^2\le \frac14 |\wntauminus|^2 + \frac14 |2\wntauminus-\wntauminuss|^2 \\[2mm]
&	 \nonumber  + \frac14 \langle B\untauminus,\untauminus\rangle + \frac14 \langle B(2\untauminus-\untauminuss),2\untauminus-\untauminuss\rangle  \\[2mm]
&	+\tau \left(\beta +\left(\frac12 +d\right)\bar{c}_{\iota}(\varepsilon) \right) |\wntau|^2 +\frac12 \tau  d^2 + C(\varepsilon)\tau \|\fntau\|_{V^*}^2 \label{ap12}
\end{align}

\noindent We replace the index $n$ by $k$ in \eqref{ap12} and sum it up for $k=2,\dots,N$. Moreover, we add (\ref{ap6}) to the resulting inequality, which leads to the following conclusion
\begin{align}
&	\nonumber  \frac14 |\wntau|^2 + \frac14 |2\wntau-\wntauminus|^2 +\frac14 \sum_{k=2}^n |\wktau-2\wktauminus+\wktauminuss|^2  \\[2mm]
&	\nonumber  +\frac14 \langle B\untau,\untau\rangle + \frac14 \langle B(2\untau-\untauminus),2\untau-\untauminus\rangle  \\[2mm]
&	\nonumber  +\frac14 \sum_{k=2}^n \langle B(\uktau-2\uktauminus+\uktauminuss),\uktau-2\uktauminus+\uktauminuss\rangle   \\[2mm]
&	\nonumber + 
\frac12 \tau \left(\alpha - \left(\frac32+d\right)\varepsilon\right)\sum_{k=1}^n \|\wktau\|^2 +\frac14 |\wone-\wzero|^2 + \frac14 \langle B(\uone-\uzero),\uone-\uzero\rangle  \\[2mm]
&	\nonumber \le \frac14 |\wzero|^2 + \frac14 \langle B\uzero,\uzero\rangle  +\frac14 |2\wone-\wzero|^2+ \frac14 \langle B(2\uone-\uzero), 2\uone-\uzero\rangle \\[2mm]
&	  + \frac12d^2 n \tau + \tau C(\varepsilon) \sum_{k=1}^n \|\fktau\|_{V^*}^2+ \tau \sum_{k=1}^n \left(\beta + \left(\frac12 +d\right)\bar{c}_\iota(\varepsilon)\right) |\wktau|^2 .\label{ap14}
\end{align}

\noindent Lets take $\varepsilon$ small enough, such that $ (\alpha - (\frac32 +d)\varepsilon) >0$. Then, from \eqref{ap14}, it follows that
\begin{align}
&	\nonumber  \frac14 |\wntau|^2 \le \frac14 |\wzero|^2 + \frac14 \langle B\uzero, \uzero\rangle  +\frac14 |2\wone-\wzero|^2 + \frac14 \langle B(2\uone-\uzero),(2\uone-\uzero)\rangle  \\[2mm]
&	\nonumber  + \frac12 d^2n\tau + \tau C(\varepsilon) \sum_{k=1}^n \|\fktau\|_{V^*}^2 + \tau \sum_{k=1}^n \left(\beta + \left(\frac12 +d\right)\bar{c}_\iota(\varepsilon)\right) |\wktau|^2 \\[2mm]
	&\label{ap15}=: C_n + \tau \sum_{k=1}^n \left(\beta + \left(\frac12 +d\right)\bar{c}_\iota(\varepsilon)\right) |\wktau|^2.
\end{align}
\noindent
Usig this notation, we obtain
\begin{equation}
	\label{ap16}|\wntau|^2 \le 4C_n + \tau \sum_{k=1}^n C |\wktau|^2
\end{equation}
and, by discrete Gronwall lemma, it follows from \eqref{ap16} that
\begin{equation}\label{ap16e}
	\max_{1\le n\le N} |\wntau|^2 \le C \max_{1\le n \le N} 4 C_n.
\end{equation}
We now need to show that $C_n$ is bounded independently of $\tau$. Indeed, $|\wzero|$ is bounded by (\ref{wzero}). The term $\langle B\uzero, \uzero\rangle$ is bounded by $H(B)$ and (\ref{uzero}). Clearly, $n\tau\leq T$ and
\begin{align}
	& \tau \sum_{k=1}^n \|\fktau\|_{V^*}^2\leq C\|f\|^2_{\cV^*}\label{ap16a}.	
\end{align}
\noindent Moreover, we have
\begin{align}
&	\frac14|2\wone-\wzero|^2 = \frac14|\wone+\wone-\wzero|^2 \le \frac12(|\wone|^2 + |\wone-\wzero|^2)\le \frac12|\wone|^2+|\wone-\wzero|^2. \label{ap17}
\end{align}
On the other hand, from (\ref{ap6}), we have for $\tau$ small enough 
\begin{align*}
	|\wone|^2+|\wone-\wzero|^2\leq C+2\tau \left(\beta +\left(\frac12 +d\right)\bar{c}_{\iota}(\varepsilon) \right) |\wone|^2\leq C+\frac12|\wone|^2.
\end{align*} 
Hence, 
\begin{align}\label{ap13}
\frac12|\wone|^2+|\wone-\wzero|^2\leq C.	
\end{align}
Combining it with (\ref{ap17}), we conclude that $\frac14|2\wone-\wzero|^2$ is bounded. It remains to deal with the term $\langle B(2\uone-\uzero),2\uone-\uzero\rangle$. Taking $\varepsilon$ small enough in (\ref{ap6}), we get
\begin{align}\label{ap16c}
	\tau\|\wone\|^2\leq C.
\end{align}  
It follows from (\ref{uzero}) and (\ref{ap16c}) that
\begin{align}
	\|\uone\|^2=\|\uzero+\tau\wone\|^2 \le 2(\|\uzero\|^2+\tau^2\|\wone\|^2 )\leq C.\label{ap19}
\end{align}
Now, using (\ref{ap16c}) and (\ref{ap19}), we estimate
\begin{align*}
&	\nonumber  \langle B(2\uone-\uzero),2\uone-\uzero\rangle \le \|B\|||2\uone-\uzero\|^2=\|B\|||\uone+\uone-\uzero\|^2  \\[2mm]
&	\le 2\|B\|(\|\uone\|^2 + \|\uone-\uzero\|^2)= 2\|B\|(\|\uone\|^2 + \tau^2 \|\wone\|^2)\leq C. 
\end{align*}

\noindent  Concluding, the right hand side of (\ref{ap16e}) is bounded, which, together with (\ref{wzero}), completes the proof of  (\ref{a2}).
 Furthermore, due to (\ref{a2}), we estimate
\begin{align}
& \tau \sum_{k=1}^n \left(\beta + \left(\frac12 +d\right)\bar{c}_\iota(\varepsilon)\right) |\wktau|^2\leq \left(\beta + \left(\frac12 +d\right)\bar{c}_\iota(\varepsilon)\right) \tau N \cdot \max_{n=1,...,N}|\wntau|\leq TC.\label{ap20}
\end{align}
We conclude that the right hand side of (\ref{ap14}) is bounded, which, together with (\ref{-1}), allows us to derive (\ref{a1}) and (\ref{a6}). 
By $H(j)$(ii) we imediately get (\ref{a3}) from \eqref{a1}.\\

We now pass to the proof of \eqref{a7}. First, observe that from telescopic principle,
\begin{align}\label{telescop}
\frac12 \uzero - \frac32 \uone - \frac12 \untauminus + \frac32 \untau = \tau \sum_{k=2}^n \wktau
\end{align}
for $n=2,\dots,N$. 
Using notation $S_n = \frac23 \tau \sum_{k=2}^n \wktau$, we have
\begin{align}
	\nonumber \untau = &S_n + \frac13 S_{n-1} + \left(\frac13\right)^2 S_{n-2} + \ldots + \left(\frac13\right)^{n-2}S_2 \\[2mm]
	& + \frac12 \bigg(3-\left(\frac13\right)^{n-1}\bigg) \uone - \frac12 \bigg(1-\left(\frac13\right)^{n-1}\bigg)\uzero.\label{ap21}
\end{align} 
We denote $\bar{S}_n = \frac23 \tau \sum_{k=2}^n \|\wktau\|$. Since $\bar{S}_j \le \bar{S}_n$ for $j\le n$, we have from \eqref{ap21}
\begin{align}
&	\nonumber\|\untau\| \le \bar{S}_n \bigg(1+\frac13 + \ldots, + \left(\frac13\right)^{n-2}\bigg) + \frac12\bigg(3-\left(\frac13\right)^{n-1}\bigg) \|\uone\| + \frac12 \bigg(1-\left(\frac13\right)^{n-1}\bigg) \|\uzero\|  \\[2mm]
	&=\frac{1-\left(\frac13\right)^{n-1}}{1-\frac13} \bar{S}_n + \frac12\bigg(3-\left(\frac13\right)^{n-1}\bigg) \|\uone\| + \frac12 \bigg(1-\left(\frac13\right)^{n-1}\bigg) \|\uzero\|. \label{ap22}
\end{align} 
\noindent Thus, \eqref{ap22} implies
\begin{equation*}
	\|\untau\| \le C\left(\tau \sum_{k=2}^n \|\wktau\|+ \|\uone\|+ \|\uzero\|\right), 
\end{equation*}
which in turn implies
\begin{align}
	\nonumber \|\untau\|^2 &\le 3C\left(\tau^2 \left(\sum_{k=2}^n \|\wktau\|   \right)^2  +\|\uone\|^2+ \|\uzero\|^2 \right)  \\[2mm]
	&\le3C\left(\tau^2 n \sum_{k=2}^n \|\wktau\|^2 + \|\uone\|^2 + \|\uzero\|^2\right) \le 
	3C\left(T \tau \sum_{k=2}^N \|\wktau\|^2 +  \|\uone\|^2 + \|\uzero\|^2\right). \label{ap23}
\end{align}

\noindent We now conclude \eqref{a7} from \eqref{ap23}, (\ref{a1}), (\ref{ap19}) and (\ref{uzero}).

From (\ref{1}), $H(A)$(ii), $H(B)$ and $H(j)$(ii) we get 
\begin{align}
	\nonumber \left\|\frac{1}{\tau}(\wone - \wzero)\right\|_{V^*} &\le \|\fone\|_{V^*} + \|A\wone\|_{V^*} + \|B\uone\|_{V^*} + \|\iota^*\xione\|_{V^*}\le \|\fone\|_{V^*} + a+b\|\wone\| \\[2mm]
	& + \|B\|\|\uone\| + \|\iota^*\|d(1+\|\iota\|\|\wone\|) \leq C(1+\|\wone\| + \|\uone\| + \|\fone\|_{V^*}). \label{z1}
\end{align}
Squaring both sides of \eqref{z1} and multiplying  by $\tau$, we obtain
\begin{equation*}
	\tau \left\|\frac{1}{\tau}(\wone-\wzero)\right\|_{V^*}^2 \le C(\tau + \tau \|\wone\|^2 + \tau \|\uone\|^2 + \tau \|\fone\|_{V^*}^2).
\end{equation*}
From (\ref{ap16c}), (\ref{ap19}) and (\ref{ap16a}), the right hand side is bounded, which proves \eqref{a4}. Similarly, \eqref{a5} follows from (\ref{2}), \eqref{a1}, \eqref{a7} and (\ref{ap16a}). 
\end{proof}

\begin{lemma}\label{Lemma_w0_w1}The following convergence holds
\[
\wone - \wzero \to 0\quad \textrm{strongly in} \,\, H.
\]
\end{lemma}
\begin{proof}
From \eqref{ap13} we have $\frac14 |\wone-\wzero|^2 \le C$, so the sequence $\{\wone-\wzero\}$ is bounded in $H$, hence, for a subsequence $\wone - \wzero \to \zeta$ weakly in $H$ and, in consequence, weakly in $V^*$ for some $\zeta \in H$. From \eqref{a4}, it follows that
\begin{align*}
\|\wone - \wzero\|_{V^*}^2 \le\tau C \to 0.
\end{align*}
Thus, $\wone-\wzero \to 0$ strongly in $V^*$, so, also weakly in $V^*$. From the uniqueness of the weak limit we get $\zeta=0$, so 
\begin{equation}\label{poprawki_19}
\wone-\wzero \to 0 \quad \textrm{weakly in } H.
\end{equation}
We now pass to the strong convergence. Clearly, $0\leq |\wone- \wzero|^2$, hence,
\begin{align}\label{poprawki_5}
	0\leq\liminf |\wone- \wzero|^2\leq \limsup |\wone- \wzero|^2.
\end{align}  
Next, we write
\begin{align}\label{poprawki_2}
|\wone- \wzero|^2 = ( \wone-\wzero,\wone-\wzero)= ( \wone-\wzero,\wone) + ( \wzero-\wone,\wzero)	
\end{align}
and calculate
\begin{align}\label{poprawki_6}
\limsup |\wone- \wzero|^2\leq \limsup ( \wone-\wzero,\wone) + \limsup( \wzero-\wone,\wzero).
\end{align}
From (\ref{poprawki_19}) and (\ref{wzero}) we have
\begin{align}\label{poprawki_3}
	\limsup (\wzero-\wone,\wzero)=\lim(\wzero-\wone,\wzero) = 0.
\end{align}
Next, we will show that 
\begin{align}\label{poprawki_4}
\limsup(\wone-\wzero,\wone) \leq 0.	
\end{align}
To this end, we apply \eqref{1} and obtain 
\begin{equation}
( \wone-\wzero,\wone)  = \tau \left(\langle \fone,\wone\rangle - \langle A\wone,\wone\rangle - \langle B\uone,\wone\rangle - \langle \iota^*\xione,\wone\rangle \right). \label{z2}
\end{equation}
Hence,
\begin{align}\label{poprawki_7}
\limsup(\wone-\wzero,\wone)\leq &\limsup \tau \langle \fone,\wone\rangle -\liminf \left(\tau \langle A\wone,\wone\rangle+\tau \langle \iota^*\xione,\wone\rangle\right)\nonumber\\[2mm]
&-\liminf (\tau \langle B\uone,\wone\rangle).
\end{align}
We estimate all three terms in the right hand side of \eqref{poprawki_7}. First,
\begin{align}
&\tau |\langle \fone,\wone\rangle| \le \tau \|\fone\|_{V^*} \|\wone\| = (\tau \|\fone\|^2_{V^*})^{1/2} (\tau \|\wone \|^2)^{1/2}. \label{z3}
\end{align}
From Jensen inequality and $H(f)$ we have 
\begin{align}\label{poprawki_20}
\tau \|\fone\|^2_{V^*} \le \int_0^\tau \|f(t)\|^2_{V^*}\, dt\to 0.	
\end{align}
From (\ref{ap16c}), (\ref{poprawki_20})  and (\ref{z3}) we obtain $\lim\tau \langle \fone,\wone\rangle=0$.\\ 
\noindent Next, combining $H(A)$(iii) with (\ref{ap4}) and taking $\varepsilon>0$ small enough, we get 
\begin{align}
&\tau \langle A\wone,\wone\rangle +\tau \langle \iota^*\xione,\wone\rangle \ge \tau\left(\alpha-\left(\frac12 +d\right) \varepsilon\right) \|\wone\|^2- \tau\left(\beta+\left(\frac12 +d\right) \bar{c}_\iota(\varepsilon)\right) |\wone|^2 \nonumber\\[2mm]
&  -\frac12 \tau  d^2\geq-\tau\left(\beta+\left(\frac12 +d\right) \bar{c}_\iota(\varepsilon)\right) |\wone|^2 -\frac12 \tau  d^2=:-\tau C |\wone|^2-\frac12 \tau  d^2 \label{z4}.
\end{align}
From (\ref{ap13}) we get  $\tau C  |\wone|^2  \to 0$ and, clearly, $\tau \frac12 d^2\to 0$. Hence, from (\ref{z4}), we have 
\begin{align*}
\liminf (\tau \langle A\wone,\wone\rangle +\tau \langle \iota^*\xione,\wone\rangle)\geq \liminf\left(-\tau C |\wone|^2- \frac12 \tau  d^2\right)=0. 	
\end{align*}
Lastly,
\begin{align}
&|\tau\langle B\uone,\wone\rangle |  \leq \sqrt{\tau} \|B\| \|\uone\|\sqrt{\tau}\|\wone\|.  \label{z5}
\end{align}
Applying (\ref{ap16c}) and (\ref{ap19}) in (\ref{z5}), we get $\lim \tau\langle B\uone,\wone\rangle=0$. Now, (\ref{poprawki_4}) follows from (\ref{poprawki_7}). 
From (\ref{poprawki_6})-(\ref{poprawki_4}) we obtain $\limsup|\wone-\wzero|^2\leq 0$. 
Combining it with (\ref{poprawki_5}), we get $\lim |\wone- \wzero|^2=0$, which completes the proof. 
\end{proof}\\

Let $I_n=((n-1)\tau, n\tau]$, $n=1,\ldots, N$. We define the functions $\bar{u}_\tau,\,\bar{w}_\tau\colon (0,T]\to V$, $\bar{\xi}_\tau\colon(0,T]\to U^*$, $\bar{f}_\tau\colon(0,T]\to V^*$, $u_\tau,\, w_\tau\colon [0,T]\to V$   by $\bar{u}_\tau(t) = \untau,\, \bar{w}_\tau (t)= \wntau, \,\bar{\xi}_\tau(t) = \xintau, \,\bar{f}_\tau(t) = \fntau$ for $t\in I_n, \ n=1,\ldots, N$, 
\begin{align*}
u_\tau(t) = \uzero + \int_0^t \bar{w}_\tau(s)\, ds	\quad \text{for}\,\,t\in[0,T]
\end{align*}
and
\begin{equation}\label{w_tau}
\displaystyle
w_\tau(t) = \begin{cases}
 \frac32 \wone - \frac12 \wzero + (\wone-\wzero) \frac{t-t_1}{\tau}  &\textrm{for} \quad t\in [0,\tau], \\[2mm]
\frac32\wntau -\frac12\wntauminus + (\frac32\wntau-2\wntauminus + \frac12\wntauminuss)\frac{t-t_n}{\tau} \quad &\textrm{for} \quad t\in I_n, \ n=2,\ldots,N.
\end{cases}
\end{equation}
We note that 
\begin{align}\label{poprawki_21}
w_\tau(0) = \frac32\wone-\frac12\wzero -(\wone-\wzero) =\frac12 \wone +\frac12 \wzero 
\end{align}
and $w_\tau(t_1)$ is well defined, since
\begin{align*}
&w_\tau(t_{1-}) = \frac32 \wone - \frac12 \wzero, \\[2mm]
&w_\tau(t_{1+}) = \frac32 w_\tau^2-\frac12 \wone - \left(\frac32 w_\tau^2-2\wone+\frac12\wzero\right) = \frac32\wone - \frac12 \wzero.
\end{align*}
Moreover, it is clear that
\begin{align}\label{poprawki_23}
u_\tau'=\bar{w}_\tau
\end{align}
and 
\begin{equation}
	\label{c13}\bar{f}_\tau \to f \quad \mbox{strongly in} \ \cV^*. 
\end{equation}
The next lemma provides the appropriate bounds for the above functions.   
\begin{lemma}\label{apriori2}
	The following bounds hold
\begin{align}
&\|\bar{u}_\tau\|_{L^\infty(0,T;V)} \le C, \label{e9}\\[2mm] 	
&\|u_\tau\|_{L^\infty(0,T;V)} \le C, \label{e10}\\[2mm] 	
&\|\bar{w}_\tau\|_{\cV}\le C, \label{e2}\\[2mm]
&\|\bar{w}_\tau\|_{L^\infty(0,T;H)} \le C, \label{e3}\\[2mm]
&\|w_\tau\|_{\cV} \le C \label{e1},\\[2mm]
&\|w_\tau\|_{L^\infty(0,T;H)} \le C, \label{e7}\\[2mm] 
&\|w_\tau'\|_{\cV^*} \le C,\label{e4}\\[2mm]
&\|\bar{w}_\tau\|_{M^{2,2}(0,T;V,V^*)} \le C, \label{e6}\\[2mm]
&\|\bar{\xi}_\tau\|_{\cU^*} \le C.\label{e8}
\end{align}
\end{lemma}

\noindent \begin{proof}
 Estimates \eqref{e9}, \eqref{e2}, \eqref{e3} and (\ref{e8}), respectively, follow immediately from \eqref{a7} and \eqref{a1}-(\ref{a3}), respectively. 
We now prove (\ref{e1}). Applying formula (\ref{w_tau}) and estimate (\ref{a1}), we calculate 
\begin{align}
&\nonumber \|w_\tau\|_{\cV}^2 = \int_0^T \|w_\tau (t)\|^2\, dt =\int_0^\tau \left\|\frac32 \wone-\frac12\wzero + \left(\wone-\wzero\right)\frac{t-t_1}{\tau}\right\|^2\, dt  \\[2mm]
&\nonumber+\sum_{n=2}^N \int_{(n-1)\tau}^{n\tau} \left\|\frac32\wntau-\frac12 \wntauminus + \left(\frac32 \wntau-2\wntauminus + \frac12\wntauminuss\right)\frac{t-t_n}{\tau}\right\|^2\, dt \\[2mm]
&\le C\int_0^\tau \|\wone\|^2 + \|\wzero\|^2 + \left|\frac{t-t_1}{\tau}\right| (\|\wone\|^2 + \|\wzero\|^2)+\sum_{n=2}^N C\tau (\|\wntau\|^2 + \|\wntauminus\|^2 + \|\wntauminuss\|^2)\nonumber\\[2mm]
& \nonumber \le C(\tau\|\wone\|^2 + \tau\|\wzero\|^2)+3 C \tau \sum_{n=0}^N \|\wntau\|^2 = C\tau \sum_{n=0}^N \|\wntau\|^2\leq C. 
\end{align}
This proves \eqref{e1}. We now prove (\ref{e10}). To this end, for $t\in (0,T)$, we estimate
\begin{align*}
	\|u_\tau(t)\|\leq \|u_\tau^0\|+\int_0^t\|\bar{w}_\tau(s)\|\,ds\leq\|u_\tau^0\|+\int_0^T\|\bar{w}_\tau(s)\|\,ds\leq \|u_\tau^0\|+\sqrt{T}\|\bar{w}_\tau\|_{\cV}. 
\end{align*}
Hence, (\ref{e10}) follows from \eqref{uzero} and \eqref{e2}.\\
 Applying formula (\ref{w_tau}) and using (\ref{a4})-(\ref{a5}), we estimate
\begin{align}\label{poprawki_8}
	\|w_\tau'\|^2_{L^2(0,T;V^*)}&=\int_0^\tau\|w'_\tau(t)\|_{V^*}^2\,dt+\sum_{n=2}^N\int_{(n-1)\tau}^{n\tau}\|w'_\tau(t)\|^2_{V^*}\,dt\nonumber\\[2mm]
	&=\tau \left\|\frac{\wone-\wzero}{\tau}\right\|_{V^*}^2+\tau\sum_{n=2}^N\left\|\frac{1}{\tau}\left(\frac32 \wntau-2\wntauminus + \frac12 \wntauminuss\right)\right\|_{V^*}^2\leq C,
\end{align}
which proves \eqref{e4}. We now  pass to the proof of \eqref{e6}. Taking into account \eqref{e2} it is enough to estimate $\|\bar{w}_\tau\|_{BV^2(0,T;V^*)}$. We know from \cite{Kalita2013} that
\begin{equation}
\label{es2}\|\bar{w}_\tau\|_{BV^2(0,T;V^*)}\le T\tau \sum_{n=1}^N \left\|\frac{\wntau-\wntauminus}{\tau}\right\|_{V^*}^2.
\end{equation}
We denote $\delta_n=\frac{1}{\tau}(\wntau-\wntauminus)$ \,\, for $ \ n=1,\ldots,N$. Then, from Lemma~\ref{BWK}
\begin{align}
&\label{es3}\tau\sum_{n=1}^N\left\|\frac{\wntau-\wntauminus}{\tau}\right\|_{V^*}^2 \\[2mm]
&\nonumber\le C\left(\tau\left\|\frac{\wone-\wzero}{\tau}\right\|^2_{V^*} + \tau\sum_{n=2}^N \left\|\frac1\tau \left(\frac32(\wntau-\wntauminus) -\frac12(\wntauminus  -w_{\tau}^{n-2}\right)\right\|^2_{V^*} \right)
 \\[2mm]
&\nonumber=C\left(\tau\left\|\frac{\wone-\wzero}{\tau}\right\|^2_{V^*} +\tau\sum_{n=2}^N \left\|\frac1\tau\left(\frac32 \wntau - 2\wntauminus +\frac12 w_\tau^{n-2}  \right)\right\|_{V^*}^2\right).
\end{align}
Applying \eqref{a4}, \eqref{a5} and \eqref{es3} in \eqref{es2} completes the proof of \eqref{e6}.
\end{proof}

\begin{lemma}\label{Z2} The following convergences hold
\begin{align}
&w_\tau-\bar{w}_\tau \to 0 \,\,\,\,\, \mbox{in} \,\,\,\,\, {\cV^*} \  \textrm{as} \ \tau \to 0,\label{convergence_1}\\[2mm]
&u_\tau - \bar{u}_\tau \to 0 \,\,\,\,\, \mbox{in} \,\,\,\,\, {\cV} \  \textrm{as} \ \tau \to 0.\label{convergence_2}
\end{align}
\end{lemma}
\begin{proof} For the proof of (\ref{convergence_1}), we calculate the difference
\begin{equation}\label{poprawki_26}
 w_\tau(t) - \bar{w}_\tau(t) = \begin{cases}
(\wone - \wzero)\frac{t-1/2 \tau}{\tau} \qquad\qquad\qquad\qquad\qquad  \text{for}\,\, t\in I_1, \\[2mm]
(\frac32 \wntau - 2\wntauminus + \frac12 \wntauminuss) \frac{t-(n-1/2)\tau}{\tau} \\[2mm]
\qquad\qquad\,\, -\frac14 (\wntau - 2\wntauminus + \wntauminuss)\  \quad\,\,\, \text{for}\,\, t\in I_n,\,n=2,\dots,N.
\end{cases}
\end{equation}
Hence, we can estimate
\begin{align*}
&\|w_\tau-\bar{w}_\tau\|_{L^2(0,T;V^*)}^2 = \int_0^\tau \|w_\tau(t)-\bar{w}_\tau(t)\|_{V^*}^2\,dt + \sum_{n=2}^N \int_{(n-1)\tau}^{n\tau} \|w_\tau(t) -\bar{w}_\tau(t)\|_{V^*}^2 \, dt  \\[2mm]
&\le \int_0^\tau \|\wone-\wzero\|_{V^*}^2 \left|\frac{t-1/2\tau}{\tau}\right|^2\,dt + 2\sum_{n=2}^N  \int_{(n-1)\tau}^{n\tau} \left\|\frac32 \wntau-2\wntauminus+\frac12 \wntauminuss\right\|_{V^*}^2 \left|\frac{t-(n-1/2)\tau}{\tau}\right|^2\,dt  \\[2mm]
&+\frac{1}{16}\cdot 2 \sum_{n=2}^N \int_{(n-1)\tau}^{n\tau} \|\wntau-2\wntauminus + \wntauminuss\|_{V^*}^2\, dt \le \frac{1}{12} \tau \|i^*\||\wone-\wzero|^2 \\[2mm]
& + \frac{1}{6} \tau\sum_{n=2}^N  \left\|\frac32\wntau-2\wntauminus+\frac12 \wntauminuss\right\|_{V^*}^2 + \frac18 \tau \|i^*\|\sum_{n=2}^N |\wntau-2\wntauminus+\wntauminuss|^2.
\end{align*}
 From Lemma~\ref{Lemma_w0_w1}, (\ref{a5}) and (\ref{a6}), the right hand side converges to $0$,
which completes the proof of (\ref{convergence_1}).\\
We now pass to the proof of (\ref{convergence_2}). First, for $t\in I_1$,
\begin{align}
&u_\tau(t)=\uzero + \int_0^t \bar{w}_\tau(s)\, ds = \uzero + \int_0^t \wone\, ds = \uzero +t\wone.
\end{align}
So, we can estimate 
\begin{align*}
&\|u_\tau(t)-\bar{u}_\tau(t)\|^2 = \|\uzero+t\wone-\uone\|^2 = \left\|\uzero+t\frac{\uone-\uzero}{\tau} - \uone\right\|^2 = \left\|(\uone-\uzero)\left(\frac{t}{\tau} -1\right) \right\|^2
\end{align*}
and, applying (\ref{ap16c}), we get 
\begin{align*}
&\int_0^\tau \|u_\tau(t)-\bar{u}_\tau(t)\|^2\, dt = \int_0^\tau \left(1-\frac{t}{\tau}\right)^2\|\uone-\uzero\|^2\, dt  \\[2mm] 
&=\frac13 \tau \|\uone-\uzero\|^2 = \frac13 \tau \tau^2 \|\frac{\uone - \uzero}{\tau}\|^2 = \frac13\tau^2 \tau \|\wone\|^2.
\end{align*}
Next, for $t\in I_k$, $k=2,\ldots,N$, applying (\ref{telescop}), we have
\begin{align*}
&u_\tau(t) =\uzero + \int_0^\tau \bar{w}_\tau (s)\, ds + \sum_{i=2}^{k-1}\int_{(i-1)\tau}^{i\tau} \bar{w}_\tau (s)\, ds + \int_{(k-1)\tau}^t \bar{w}_\tau(s)\ ds  \\[2mm]
&=\uzero + \int_0^\tau \frac{\uone-\uzero}{\tau}\, ds + \tau\sum_{i=2}^{k-1}w_\tau^i + (t-(k-1)\tau)\wktau  \\[2mm]
&=\uone + \frac12 \uzero-\frac32\uone - \frac12 \uktauminuss+\frac32\uktauminus + \frac{t-(k-1)\tau}{\tau} \left(\frac32\uktau-2\uktauminus+\frac12\uktauminuss\right).
\end{align*}
So the difference can be written as
\begin{equation}
\label{z9} u_\tau(t)-\bar{u}_\tau(t) = \frac12 \uzero-\frac12\uone -\frac12\uktauminuss+\frac32\uktauminus - \uktau + \frac{t-(k-1)\tau}{\tau}\left(\frac32\uktau - 2\uktauminus +\frac12\uktauminuss\right).
\end{equation}
Hence, \eqref{z9} implies that
\begin{align*}
\|u_\tau(t)-\bar{u}_\tau(t)\|^2 &\le 
3 \left( \frac14 \|\uzero-\uone\|^2 + \left\|\uktau - \frac32\uktauminus +\frac12\uktauminuss\right\|^2 + \left\|\frac32\uktau-2\uktauminus+\frac12\uktauminuss\right\|^2\right)  \\[2mm]
&\le 3\left(\frac14\|\uzero-\uone\|^2 + \frac{17}{9} \left\|\frac32\uktau-2\uktauminus+\frac12\uktauminuss\right\|^2 + \frac{1}{18} \|\uktauminus - \uktauminuss\|^2\right) \\[2mm]
&= 3\left(\frac14 \tau^2 \|\wone\|^2 + \frac{17}{9}\tau^2 \|\wktau\|^2 + \frac{1}{18}\|\uktauminus-\uktauminuss\|^2\right).
\end{align*}
In the last estimation we have used the following argument
\begin{align*}
\left\|\uktau - \frac32\uktauminus+\frac12\uktauminuss\right\|^2 &=\frac49 \left\|\frac32\uktau-\frac94\uktauminus+\frac34\uktauminuss\right\|^2 \\[2mm]
&=\frac49 \left\|\left(\frac32\uktau-2\uktauminus+\frac12\uktauminuss\right)-\frac14 \uktauminus+\frac14\uktauminuss\right\|^2 \\[2mm]
&\le \frac49 \cdot 2 \left(\left\|\frac32\uktau-2\uktauminus+\frac12\uktauminuss\right\|^2 + \frac{1}{16} \left\|\uktauminus-\uktauminuss\right\|^2\right).
\end{align*}
Thus, we get for $t\in I_k$, $k=2,\ldots,N$.
\begin{equation}
\|u_\tau(t)-\bar{u}_\tau(t)\|^2 \le C\left(\tau^2\|\wone\|^2 + \tau^2\|\wktau\|^2 + \|\uktauminus-\uktauminuss\|^2\right).
\end{equation}
We can now estimate the full norm
\begin{align}
&\nonumber\|u_\tau-\bar{u}_\tau\|_{\cV}^2 = \int_0^\tau\|u_\tau(t)-\bar{u}_\tau(t)\|^2\, dt + \sum_{k=2}^N \int_{(k-1)\tau}^{k\tau} \|u_\tau(t)-\bar{u}_\tau(t)\|^2\, dt \\[2mm]
&\nonumber\le \frac13\tau^2\tau \|\wone\|^2 + \sum_{k=2}^N C\left(\tau^3\|\wone\|_V^2 + \tau^3 \|\wktau\|^2 + \tau \|\uktauminus-\uktauminuss\|^2\right) \\[2mm]
&\nonumber\le \frac13 \tau^2 \tau \|\wone\|^2 + CN\tau^3 \|\wone\|^2 + \tau^3 \sum_{k=2}^N \|\wktau\|^2 + \tau \sum_{k=2}^N \|\uktauminus-\uktauminuss\|^2 \\[2mm]
&\le \frac13 \tau^2 \tau \|\wone\|^2 + CT\tau \tau\|\wone\|^2 + \tau^2 \tau \sum_{k=2}^2 \|\wktau\|^2 + \tau \sum_{k=2}^N \|\uktauminus-\uktauminuss\|^2. \label{z10}
\end{align}
From Lemma~\ref{BWK} we have 
\begin{align*}
&\sum_{k=2}^N \|\uktauminus-\uktauminuss\|^2 = \sum_{k=1}^{N-1} \|\uktau-\uktauminus\|^2 \\[2mm]
&\le C\left(\sum_{i=2}^N \left\|\frac32(\uktau-\uktauminus) -\frac12 (\uktauminus-\uktauminuss)\right\|^2 + \|\uone-\uzero\|^2\right)  \\[2mm]
&= C\left(\sum_{k=2}^N \tau^2 \|\wktau\|^2 + \tau^2 \|\wone\|^2\right) = C\left( \tau \tau \sum_{k=1}^{N} \|\wktau\|^2\right) \to 0.
\end{align*}
Applying the above convergence together with  (\ref{ap16c}) and (\ref{a1}) to the right hand side of \eqref{z10}, we get the proof of (\ref{convergence_2}).
\end{proof}\\

\noindent At the end of this section we observe that \eqref{1}--\eqref{2a} imply
\begin{align}
& \label{f1}w_\tau'(t) + A\bar{w}_\tau(t) + B{u}_\tau(t) + \iota^* \bar{\xi}_\tau(t) = \bar{f}_\tau(t)  \ \ \mbox{for a.e.} \ t\in (0,T),\\[2mm]
& \label{f2}\bar{\xi}_\tau (t)\in \partial j(\iota \bar{w}_\tau(t))  \ \quad\qquad\qquad\qquad\qquad\qquad \mbox{for a.e.} \ t\in (0,T).
\end{align}

\section{Convergence}\label{Sec_5}
In this section we study a convergence of the functions introduced in the previous section to a solution of Problem ${\cal P}$.
For that purpose, we will provide  their convergence to some element in space $\cal V$ and then, passing to the limit in \eqref{f1} and \eqref{f2}, we will show that the limit element solves the original problem.  
First of all, we introduce the Nemytskii operators $\mathcal{A}\colon \cV\to \cV^*$, $\mathcal{B}\colon \cV\to\cV^*$, $\bar{\iota}\colon \cV\to \cU$ corresponding to $A$, $B$ and $\iota$, respectively, given by 
\[
(\mathcal{A}v)(t) = Av(t), \ (\mathcal{B}v)(t) = Bv(t) \ \mbox{and} \ (\bar{\iota}v)(t) = \iota v(t)\qquad \mbox{for all} \ v\in \cV.
\]
Let $\bar{\iota}^*\colon \cU^*\to\cV^*$ denote the adjoint operator to $\bar{\iota}$. Then \eqref{f1} and \eqref{f2} are equivalent with
\begin{align}
&\label{c-1}w_\tau'(t) + (\mathcal{A}\bar{w}_\tau)(t) + (\mathcal{B}\bar{u}_\tau)(t)+ (\bar{\iota}^*\bar{\xi}_\tau)(t) = \bar{f}_\tau(t) \quad\,\,\, \mbox{for a.e.} \ t\in (0,T),\\[2mm]
&\label{c0}\bar{\xi}_\tau(t)\in \partial j((\bar{\iota}\bar{w}_\tau)(t))  \qquad\qquad\qquad\qquad\qquad\qquad\quad\,\,\,\, \mbox{for a.e.} \ t\in (0,T).
\end{align}

\noindent In what follows, we will use the following pseudomonotonicity property of operator ${\cal A}$ (see Lemma 1 in \cite{Kalita2013}).

\begin{lemma}\label{lemma:Bartosz2}
	Assume that $H(A)$ holds and a sequence $\{v_n\}\subset {\cal V}$ satisfies: $v_n$ is bounded in $M^{2,2}(0,T;V,V^*)$, $v_n\to v$ weakly in ${\cal V}$ and $\limsup_{n\to\infty}\dual{{\cal A}v_n,v_n-v}{\cal V}\leq 0$. Then ${\cal A}v_n\to {\cal A}v$ weakly in ${\cal V^*}$. 
\end{lemma}

\noindent We are now in a position to provide the main convergence result. 
\begin{theorem}\label{theorem_1}
	There exist functions $u\in\cV$ and $w\in\mathcal{W}$ such that 
	\begin{align}
	& \label{c-2} u_\tau \to u \ \mbox{weakly in} \ \mathcal{V} \ \mbox{and weakly* in} \ L^\infty(0,T;V), \\[2mm]
&\label{c-3}\bar{u}_\tau \to u \ \mbox{weakly in} \ \mathcal{V} \ \mbox{and weakly* in} \ L^\infty(0,T;V),\\[2mm]
	& \label{c1} w_\tau \to w \ \mbox{weakly in} \ \mathcal{W} \ \mbox{and weakly* in} \ L^\infty(0,T;H), \\[2mm]
	&\label{c2}\bar{w}_\tau \to w \ \mbox{weakly in} \ \mathcal{V} \ \mbox{and weakly* in} \ L^\infty(0,T;H).
	\end{align}
Moreover, $w=u'$ and $u$ is a solution of Problem ${\cal P}$.
\end{theorem}
\begin{proof}
First, we observe that the bounds (\ref{e9}) and (\ref{e10}) imply that $\|\bar{u}_\tau\|_{\cV}, \|{u}_\tau\|_{\cV}\leq C$. From these and the other bounds 	 obtained in Lemma~\ref{apriori2} there exist $u_1, u_2\in \cV$,  $w_1, w_2\in \cV$, $w_3\in\cV^*$ and $\xi\in\cU^*$, such that
	\begin{align}
&\label{c27}\bar{u}_\tau \to u_1 \ \mbox{weakly in} \ \mathcal{V} \ \mbox{and weakly* in} \ L^\infty(0,T;V), \\[2mm]
&\label{c28}u_\tau \to u_2 \ \mbox{weakly in} \ \mathcal{V} \ \mbox{and weakly* in} \ L^\infty(0,T;V), \\[2mm]
	&\label{c3}\bar{w}_\tau \to w_1 \ \mbox{weakly in} \ \mathcal{V} \ \mbox{and weakly* in} \ L^\infty(0,T;H), \\[2mm]
	&\label{c4}w_\tau \to w_2 \ \mbox{weakly in} \ \mathcal{V} \ \mbox{and weakly* in} \ L^\infty(0,T;H), \\[2mm]
	&\label{c5}w_\tau'\to w_3 \ \mbox{weakly in} \ \mathcal{V}^*,\\[2mm]
	&\label{c6}\bar{\xi}_\tau \to \xi \ \mbox{weakly in} \ \mathcal{U}^*.
	\end{align}

\noindent From (\ref{c27})-(\ref{c28}) we have $\bar{u}_\tau-u_\tau\to u_1-u_2$ weakly in $\cV$. On the other hand, from (\ref{convergence_2}), $\bar{u}_\tau-u_\tau\to 0$ weakly in $\cV$. By the uniqueness of the weak limit we have $u_1-u_2=0$, hence, $u_1=u_2=:u$, which proves (\ref{c-2}) and (\ref{c-3}). In a similar way, from (\ref{c3})-(\ref{c4}) and (\ref{convergence_1}), we obtain $w_1=w_2=:w$, which proves (\ref{c2}) and, partially, (\ref{c1}), namely, ${w}_\tau\to w$ weakly in $\cV$. From (\ref{c4})-(\ref{c5}), by standard arguments, we have $w_3=w_2'=w'$, which completes the proof of (\ref{c1}).\\
Now, using (\ref{poprawki_23}), we conclude from  (\ref{c28})-(\ref{c3}) that $u_\tau\to u$ weakly in $\cV$ and $\bar{w}_\tau=u_\tau'\to w$ weakly in $\cV$. Hence, by standard argument, we have $w=u'$. \\

 Now, we will show that $u$ is a solution of Problem ${\cal P}$. To this end, we deal with initial condition first. From (\ref{c28})-(\ref{c3}) we have $u_\tau\to u$ weakly in $\cW$. As the embedding $\cW\subset C([0,T];H)$ is continuous, it follows that, for a subsequence, $u_\tau\to u$ weakly in $C([0,T];H)$. Hence, in particular, $u_\tau(0)\to u(0)$ weakly in $H$. On the other hand, from (\ref{uzero}), we have $u_\tau(0)=u_\tau^0\to u_0$ strongly in $V$ and, in consequence, weakly in $H$. Summarizing, by the uniqueness of the weak limit, we have
 \begin{align}\label{poprawki_10}
 	u_\tau(0)\to u(0)=u_0 \quad \text{strongly in}\,\,V.
 \end{align}   
  From \eqref{c4}-\eqref{c5} we have $w_\tau\to w$ weakly in $\cW$. As the embedding $\cW\subset C([0,T];H)$ is continuous, it follows that, for a subsequence, $w_\tau\to w$ weakly in $C(0,T;H)$. Hence, in particular, 
\begin{align}\label{c7}
w_\tau(0)\to w(0) \quad \mbox{weakly in} \ H.
\end{align}
On the other hand, from (\ref{poprawki_21}), we have
\begin{align*}
w_\tau(0)=\frac12w_\tau^1+\frac12 w_\tau^0=\frac12(w_\tau^1-w_\tau^0)+w_\tau^0.
\end{align*}
Hence, by Lemma \ref{Lemma_w0_w1} and by \eqref{wzero}, we get $w_\tau(0)\to w_0$ strongly in $H$, so also weakly in $H$. Comparing it with (\ref{c7}), we conclude, from uniqueness of the weak limit in $H$, that
\begin{align}\label{c8}
w_\tau(0)\to w_0=w(0) \quad \mbox{strongly in} \ H.
\end{align}
\noindent We now turn to \eqref{c-1}. Testing it with an arbitrary element $v\in\cV$, we obtain
\begin{equation}
\label{c9}\langle w_\tau',v\rangle_{\cV^*\times \cV} + \langle \mathcal{A}\bar{w}_\tau, v\rangle_{\cV^*\times \cV} +   \langle\mathcal{B}\bar{u}_\tau,v\rangle_{\cV^*\times \cV} + \langle \bar{\iota}^*\bar{\xi}_\tau, v \rangle_{\cV^*\times \cV}= \langle \bar{f}_\tau, v \rangle_{\cV^*\times \cV}.
\end{equation}
Our goal is to pass to the limit with \eqref{c9} as $\tau\to 0$.
From \eqref{c1} we have
\begin{equation}
\label{c10}\langle w_\tau',v\rangle_{\cV^*\times \cV} \to \langle w',v\rangle_{\cV^*\times \cV}.
\end{equation}
As the Nemytskii operator $\mathcal{B}$ is linear and continuous, it is also weakly continuous. Hence, it follows from (\ref{c-3}) that
\begin{equation}
	\langle\mathcal{B}\bar{u}_\tau,v\rangle_{\cV^*\times \cV}\to \langle\mathcal{B}u,v\rangle_{\cV^*\times \cV}.
\end{equation}
From \eqref{c6} we have
\begin{equation}
	\langle \bar{\iota}^*\bar{\xi}_\tau, v \rangle_{\cV^*\times \cV}=\langle \bar{\xi}_\tau, \bar{\iota}v \rangle_{\cU^*\times \cU}\to\langle \xi, \bar{\iota}v \rangle_{\cU^*\times \cU}=\langle \bar{\iota}^*\xi, v \rangle_{\cV^*\times \cV}.
\end{equation}
From (\ref{c13}), it follows that $\bar{f}_\tau\to f$ weakly in $\cV^*$, hence,
\begin{equation}
\label{c12}\langle \bar{f}_\tau, v \rangle_{\cV^*\times \cV} \to \langle f,v \rangle_{\cV^*\times \cV}.
\end{equation}
It remains to pass to the limit with the nonlinear, hence most problematic term $\mathcal{A}\bar{w}_\tau$. To deal with it, we will show that
\begin{equation}
\label{c14}\limsup \langle \mathcal{A}\bar{w}_\tau, \bar{w}_\tau-w\rangle_{\cV^*\times \cV} \le 0
\end{equation}
and apply Lemma \ref{lemma:Bartosz2}. From \eqref{c9}, it follows that
\begin{align*}
& \langle \A \bar{w}_\tau,\bar{w}_\tau -w\rangle_{\cV^*\times \cV} = \langle \bar{f}_\tau, \bar{w}_\tau-w\rangle_{\cV^*\times \cV} - \langle \B \bar{u}_\tau,\bar{w}_\tau -w\rangle_{\cV^*\times \cV}- \langle w_\tau', \bar{w}_\tau-w \rangle_{\cV^*\times \cV} - \langle \bar{\xi}_\tau,\bar{\iota} \bar{w}_\tau-\bar{\iota} w\rangle_{\cU^*\times \cU}.
\end{align*}
Hence,
\begin{align}
\nonumber\limsup \langle \A \bar{w}_\tau, \bar{w}_\tau-w\rangle_{\cV^*\times \cV} &\le \limsup \langle \bar{f}_\tau, \bar{w}_\tau-w\rangle_{\cV^*\times \cV}-\liminf \langle \B \bar{u}_\tau,\bar{w}_\tau -w\rangle_{\cV^*\times \cV} \\[2mm]
&\label{c15} -\liminf\langle w_\tau',\bar{w}_\tau-w\rangle_{\cV^*\times \cV} - \liminf\langle \bar{\xi}_\tau, \bar{\iota}\bar{w}_\tau-\bar{\iota}w\rangle_{\cU^*\times \cU}. 
\end{align}
From \eqref{c13} and \eqref{c2} we have
\begin{equation}
	\label{c16}\limsup \langle \bar{f}_\tau, \bar{w}_\tau-w\rangle_{\cV^*\times \cV}=\lim \langle \bar{f}_\tau, \bar{w}_\tau-w\rangle_{\cV^*\times \cV} =0.
\end{equation}
Next, we write
\begin{align}\label{poprawki_9}
	&\langle \B \bar{u}_\tau,\bar{w}_\tau -w\rangle_{\cV^*\times \cV}=\langle \B \bar{u}_\tau-\B u_\tau,\bar{w}_\tau -w\rangle_{\cV^*\times \cV}\nonumber\\[2mm]
	&+\langle \B {u}_\tau-\B u,\bar{w}_\tau -w\rangle_{\cV^*\times \cV}+\langle \B {u},\bar{w}_\tau -w\rangle_{\cV^*\times \cV}
\end{align}
and calculate
\begin{align*}
	\left|\langle \B \bar{u}_\tau-\B u_\tau,\bar{w}_\tau -w\rangle_{\cV^*\times \cV}\right|\leq \|\B\|\|\bar{u}_\tau-u_\tau\|_{\cV}\|\bar{w}_\tau -w\|_{\cV}\leq 
	\|\B\|\|\bar{u}_\tau-u_\tau\|_{\cV}\left(\|\bar{w}_\tau\|_{\cV} +\|w\|_{\cV}\right).
\end{align*}
Now, applying (\ref{e2}) and (\ref{convergence_2}), we get
\begin{align}\label{poprawki_11}
\lim\langle \B \bar{u}_\tau-\B u_\tau,\bar{w}_\tau -w\rangle_{\cV^*\times \cV}=0.
\end{align} 
Furthermore, from (\ref{poprawki_23}) and $H(B)$
\begin{align*}
&\langle \B {u}_\tau-\B u,\bar{w}_\tau -w\rangle_{\cV^*\times \cV}=\langle \B {u}_\tau-\B u,{u}'_\tau -u'\rangle_{\cV^*\times \cV}=\nonumber\\[2mm]
&\frac12\skalar{\B u_\tau(T)-\B u(T),u_\tau(T)-u(T)}-\frac12\skalar{\B u_\tau(0)-\B u(0),u_\tau(0)-u(0)}\nonumber\\[2mm]
&\geq -\frac12\skalar{\B u_\tau(0)-\B u(0),u_\tau(0)-u(0)}\geq -\frac12\|\B\|\|u_\tau(0)-u(0)\|^2.
\end{align*}
Now, applying (\ref{poprawki_10}), we get
\begin{align}\label{poprawki_12}
\liminf \langle \B {u}_\tau-\B u,\bar{w}_\tau -w\rangle_{\cV^*\times \cV}\geq \liminf\left(- \frac12\|\B\|\|u_\tau(0)-u(0)\|^2\right)=0.
\end{align}
Finally, from (\ref{c2}), we obtain
\begin{align}\label{poprawki_13}
	\lim \langle \B {u},\bar{w}_\tau -w\rangle_{\cV^*\times \cV}=0.
\end{align}
Coming back to (\ref{poprawki_9}) and applying (\ref{poprawki_11})-(\ref{poprawki_13}), we get
\begin{align}\label{poprawki_14}
\liminf \langle \B \bar{u}_\tau,\bar{w}_\tau -w\rangle_{\cV^*\times \cV}\geq \liminf \langle \B \bar{u}_\tau-\B u_\tau,\bar{w}_\tau -w\rangle_{\cV^*\times \cV}\nonumber\\[2mm]
+\liminf \langle \B {u}_\tau-\B u,\bar{w}_\tau -w\rangle_{\cV^*\times \cV}+\liminf \langle \B {u},\bar{w}_\tau -w\rangle_{\cV^*\times \cV}=0.
\end{align}  
Next, from \eqref{e6}, $H(\iota)$ and Proposition \ref{prop:Bartosz6}, we have for a subsequence 
\begin{equation}
\label{c17}\bar{\iota}_1\bar{w}_\tau\to \bar{\iota}_1 w \quad \mbox{in} \ L^2(0,T;Z),
\end{equation}
where $\bar{\iota}_1\colon\cV\to L^2(0,T;Z)$ is the Nemytskii operator corresponding to $\iota_1$. Hence, we get
\begin{align}
&\nonumber\|\bar{\iota}\bar{w}_\tau-\bar{\iota}w\|_{\cU}^2 = \int_0^T\|\iota\bar{w}_\tau(t) -\iota w(t)\|_U^2\, dt = \int_0^T\|\iota_2\circ \iota_1(\bar{w}_\tau(t) -w(t))\|_U^2\, dt \\[2mm]
&\label{c18} \le \|\iota_2\|^2_{{\cal L}(Z,U)} \int_0^T\|\iota_1(\bar{w}_\tau(t) -w(t))\|_Z^2 \, dt = \|\iota_2\|^2_{{\cal L}(Z,U)} \|\bar{\iota}_1 \bar{w}_\tau-\bar{\iota}_1w\|^2_{L^2(0,T;Z)}.
\end{align}
From \eqref{c17} and \eqref{c18}, it follows that
\begin{equation}
\label{c19}\bar{\iota}\bar{w}_\tau\to \bar{\iota}w \quad \text{strongly in}\,\,\, \cU.
\end{equation}
From \eqref{c6} and \eqref{c19} we have
\begin{equation}
\label{c20}\lim\langle \bar{\xi}_\tau, \bar{\iota}\bar{w}_\tau-\bar{\iota}w\rangle_{\cU^*\times \cU} =0.
\end{equation}
Now, we will show that
\begin{equation}
\label{c21}\liminf \langle w_\tau', \bar{w}_\tau-w\rangle_{\cV^*\times \cV} \ge 0.
\end{equation}
To this end, we write
\begin{equation}
\label{c22}\langle w_\tau', \bar{w}_\tau-w\rangle_{\cV^*\times \cV} = \langle w_\tau', \bar{w}_\tau-w_\tau\rangle_{\cV^*\times \cV} + \langle w_\tau'-w',w_\tau -w\rangle_{\cV^*\times \cV}+ \langle w',w_\tau -w\rangle_{\cV^*\times \cV}.
\end{equation}
Hence,
\begin{align}\label{poprawki_24}
&\nonumber\liminf\langle w_\tau', \bar{w}_\tau-w\rangle_{\cV^*\times \cV} \geq \liminf\langle w_\tau', \bar{w}_\tau-w_\tau\rangle_{\cV^*\times \cV}\\[2mm] 
&+ \liminf \langle w_\tau'-w',w_\tau -w\rangle_{\cV^*\times \cV}+ \liminf\langle w',w_\tau -w\rangle_{\cV^*\times \cV}
\end{align}
We deal with the first term of the right hand side of (\ref{poprawki_24}). Directly from (\ref{w_tau}) and (\ref{poprawki_26}) we calculate
\begin{align}
&\nonumber\langle w_\tau',\bar{w}_\tau-w_\tau \rangle_{\cV^*\times \cV} = -\int_0^\tau \frac1\tau\left(  \wone-\wzero,\wone-\wzero \right)\frac{t-\half\tau}{\tau}\, dt  \\[2mm]
&\nonumber-\sum_{n=2}^N\int_{(n-1)\tau}^{n\tau}\frac1\tau \left(\frac32 \wntau - 2\wntauminus +\half \wntauminuss,\frac32 \wntau - 2\wntauminus +\half \wntauminuss\right)\frac{t-(n-\half)\tau}{\tau}\,dt \\[2mm]
&\nonumber+\sum_{n=2}^N\int_{(n-1)\tau}^{n\tau}\frac1\tau \left(\frac32 \wntau - 2\wntauminus +\half \wntauminuss, \frac14(\wntau-2\wntauminus+\wntauminuss) \right)\,dt.
\end{align}
We observe that
\begin{equation}
\nonumber\int_0^\tau \frac1\tau\left(  \wone-\wzero,\wone-\wzero \right)\frac{t-\half\tau}{\tau}\, dt =0,
\end{equation}
and for $n=2,\dots,N$
\begin{equation}
\nonumber\int_{(n-1)\tau}^{n\tau}\frac1\tau \left(\frac32 \wntau - 2\wntauminus +\half \wntauminuss, \frac32 \wntau - 2\wntauminus +\half \wntauminuss\right)\frac{t-(n-\half)\tau}{\tau}\, dt =0.
\end{equation}
Hence, 
\begin{equation}
\label{c24}\langle w_\tau', \bar{w}_\tau -w_\tau\rangle_{\cV^*\times \cV} = \frac14\sum_{n=2}^N \left(\frac32 \wntau-2\wntauminus +\frac12\wntauminuss, \wntau-2\wntauminus + \wntauminuss\right).
\end{equation}
We now use the following property 
\begin{align*}
\left(\frac32a-2b+\frac12 c,a-2b+c\right)&=\frac12|a-b|^2+|a-2b+c|^2-\frac12 |b-c|^2\\[2mm]
&\ge \frac12|a-b|^2-\frac12 |b-c|^2
\end{align*}
for all $a,b,c\in H$. Then,
\begin{align}
&\nonumber\sum_{n=2}^N \left(\frac32 \wntau-2\wntauminus +\frac12\wntauminuss, \wntau-2\wntauminus + \wntauminuss\right) \ge\\[2mm]
&\nonumber \sum_{n=2}^N \left(\frac12 \left|\wntau-\wntauminus\right|^2 -\frac12\left|\wntauminus-\wntauminuss\right|^2\right) = \\[2mm]
&\label{c25}\frac12 \left|w_\tau^N-w_\tau^{N-1}\right|^2 - \frac12 \left|\wone-\wzero\right|^2 \ge -\frac12 \left|\wone-\wzero\right|^2.
\end{align}
Using \eqref{c24}, \eqref{c25} and Lemma~\ref{Lemma_w0_w1}, we conclude that 
\begin{align}\label{poprawki_25}
\liminf\langle w_\tau', \bar{w}_\tau -w_\tau\rangle_{\cV^*\times \cV}\ge 0.
\end{align}
To deal with the second term of the right hand side of (\ref{poprawki_24}), we will use (\ref{c8}). Namely,
\begin{align}
	&\nonumber \liminf \langle w_\tau'-w',w_\tau -w\rangle_{\cV^*\times \cV}=\liminf\int_0^T \langle w_\tau'-w',w_\tau-w\rangle\, dt\\[2mm] 
	&\nonumber= \liminf\left(\half|w_\tau(T)-w(T)|^2 -\half|w_\tau(0)-w(0)|^2\right) \\[2mm]
	& \nonumber \ge \liminf \half |w_\tau(T)-w(T)|^2 +\liminf\left(-\half|w_\tau(0)-w(0)|^2\right)\\[2mm]
	&\label{c23}=\liminf \half |w_\tau(T)-w(T)|^2 \ge 0. 
\end{align}
\noindent Finally, from (\ref{c1}), we have $\liminf\langle w',w_\tau -w\rangle_{\cV^*\times \cV}=0$, which together with \eqref{poprawki_24}, \eqref{poprawki_25} and \eqref{c23}  completes the proof of \eqref{c21}. Now, \eqref{c14} follows from \eqref{c15}, \eqref{c16}, \eqref{poprawki_14},  \eqref{c20} and \eqref{c21}. Furthermore, from \eqref{e6}, \eqref{c3}, \eqref{c14} and Lemma \ref{lemma:Bartosz2}, we get
\begin{equation}
\label{c26}\A\bar{w}_\tau \to \A w \quad \mbox{weakly in} \ \cV^*.
\end{equation}
Using \eqref{c10}--\eqref{c12} and \eqref{c26}, we can pass to the limit in \eqref{c9} and get 
\begin{equation}
\langle w' + \A w + \mathcal{B}u + \bar{\iota}^* \xi -f, v\rangle_{\cV^*\times \cV}=0,
 \end{equation}
 which, by standard technique, shows that $u$ satisfy the first relation of Problem $\cal P$.\\
 It remains to pass to the limit with (\ref{c0}).  
 First, we recall that the multifunction $\partial j\colon U\to 2^{U^*}$
 has nonempty, closed and convex values. Furthermore,
 by Proposition 5.6.10 of \cite{DMP2}, it is also upper semicontinuous from $U$ (equipped
 with the strong topology) into $U^*$ (equipped with the weak topology). Hence, from \eqref{c0}, \eqref{c6}, \eqref{c19} and Proposition \ref{prop:Bartosz7}, we have 
 \begin{equation}
 \xi(t)\in \partial j(\iota w(t)) \quad \mbox{for a.e.} \ t\in (0,T),
 \end{equation}
 which completes the proof of the theorem. 
\end{proof}

\end{document}